\documentclass[aop,MSNbibl,citesort,dvips]{arximspdf}
\usepackage{graphicx}

\doi{10.1214/11-AOP673}
\volume{40}
\issue{5}
\pubyear{2012}
\firstpage{2197}
\lastpage{2235}

\makeatletter

\newcommand{\tr}{\operatorname{tr}}
\newcommand{\rr}{\mathbb{R}}
\newcommand{\cp}{\mathbb{C}}
\newcommand{\im}{\operatorname{Im}}

\newcommand{\var}{\operatorname{Var}}
\newcommand{\ecount}{\mathcal{N}}
\newcommand{\dsc}{f_{\mathrm{sc}}}
\newcommand{\ltwo}{\mathbb{L}^2}
\newcommand{\one}{\mathbf{1}}

\newcommand{\tree}{\mathbb{T}}
\newcommand{\leaf}{\mathrm{leaf}}
\newcommand{\Root}{\mathrm{root}}

\newtheorem{theorem}{Theorem}
\newtheorem{lemma}[theorem]{Lemma}
\newtheorem{prop}[theorem]{Proposition}

\newproclaim{defn}{Definition}
\newproclaim{rmk}{Remark}

\makeatother

\begin{document}
\begin{frontmatter}

\title{Sparse regular random graphs: Spectral density and eigenvectors}
\runtitle{sparse regular random graphs}

\begin{aug}
\author[A]{\fnms{Ioana} \snm{Dumitriu}\corref{}\thanksref{t2}\ead[label=e1]{dumitriu@math.washington.edu}}
and
\author[B]{\fnms{Soumik} \snm{Pal}\thanksref{t3}\ead[label=e2]{soumik@math.washington.edu}}

\runauthor{I. Dumitriu and S. Pal}
\affiliation{University of Washington}
\address[A]{University of Washington\\
C-342 Padelford Hall\\
Seattle, Washington 98195\\
USA\\
\printead{e1}} 
\address[B]{University of Washington\\
C-547 Padelford Hall\\
Seattle, Washington 98195\\
USA\\
\printead{e2}}
\end{aug}

\thankstext{t2}{Supported by NSF CAREER Award DMS-08-47661.}
\thankstext{t3}{Supported in part by NSF Grant DMS-10-07563.}

\received{\smonth{11} \syear{2009}}
\revised{\smonth{4} \syear{2011}}

%
\begin{abstract}
We examine the empirical distribution of the eigenvalues and the
eigenvectors of adjacency matrices of sparse regular random graphs. We
find that when the degree sequence of the graph slowly increases to
infinity with the number of vertices, the empirical spectral
distribution converges to the semicircle law. Moreover, we prove
concentration estimates on the number of eigenvalues over progressively
smaller intervals. We also show that, with high probability, all the
eigenvectors are delocalized.
\end{abstract}

%

\end{frontmatter}

\section{Introduction}

Consider the uniform distribution over the space of all labeled simple
graphs on $n$ vertices where every vertex has degree $d$. We denote a
graph randomly selected from this distribution by $G(n,d)$; the
vertices of $G(n,d)$ will always be labeled by $\{1,2,\ldots,n\}$.

Now, consider a sequence of such random graphs $\{G(n,d_n), n \in
\mathbb{N} \}$ which are $d_n$-regular of order $n$. We assume $d_n$
to be slowly growing with $n$ in a manner which will be made more
precise later. Consider the adjacency matrix $A_n$ of $G(n,d_n)$; the
$(i,j)$th element of $A_n$ is one or zero depending on whether there is
an edge between vertices $i$ and $j$ in the graph $G(n,d_n)$. The
random matrix $A_n$ is always symmetric, and it has $n$ real
eigenvalues (perhaps not all distinct) and corresponding real
eigenspaces. Under appropriate conditions on the growth of the sequence
$d_n$, we study the following phenomena as $n$ tends to infinity:
\begin{longlist}[(iii)]
\item[(i)] Global semicircle law. We prove the convergence of the
empirical spectral distribution (ESD) of the scaled adjacency matrix to
the probability measure on $[-2,2]$ with density
%
%
\begin{equation}\label{whatisfsc}
\dsc(x) = \frac{1}{2\pi} \sqrt{4 - x^2},\qquad -2 < x < 2.
\end{equation}
\item[(ii)] Local semicircle law (a.k.a. the semicircle law on short
scales). We obtain concentration estimates of the deviation of the
number of eigenvalues~$\ecount_I$ that lie in a small interval $I$
from its predicted number $n \int_I \dsc(x) \,dx$. The size of $I$
will be taken to be vanishing at an appropriate rate with increasing~$n$.

\item[(iii)] Delocalization of eigenvector coordinates. We obtain
probability estimates of the event that, for some eigenvector, a few of
the coordinates are significantly larger in magnitude than the rest.
\end{longlist}

These problems connect two areas of study: Wigner random matrices and
spectra of sparse random graphs; since we have already mentioned the
latter, we will now talk about the former.

In the last couple of decades there has been an enormous amount of
activity in the study of universal properties of random matrices,
inspired by their connection to (universal) physical systems. The
literature on universality studies in random matrices is vast; we
mention here only a few references and ask the reader to look to them
for further ones.

For an introduction and motivation to the subject, we recommend Deift's
ICM address~\cite{deift06a}. Of particular interest are the Wigner
matrices (see Bai~\cite{baiconv}, Soshnikov~\cite{soshuniv}, Bai and
Yao~\cite{baiyao05}, Khorunzhy, Khoruzhenko and Pastur
\cite{khokhopastur}, Guionnet and Zeitouni~\cite{guionnetzeitouni},
Ben Arous and Peche~\cite{benarouspeche}, Tao and Vu
\cite{taovu,taovu2}). Deift and Goiev~\cite{deiftgioev} looked at different
potential functions on symmetric, Hermitian and self-dual matrices;
Baik and Suidan explored connections to percolation
\cite{baiksuidan05} and random walks~\cite{baiksuidan07};
$\beta$-generalizations of the classical ensembles and universal
properties thereof have been explored in Forrester and Baker
\cite{Forresterpoly}, Johansson~\cite{johanssoncltherm}, Dumitriu
and Edelman~\cite{dumitriu06}. Recently, more
sophisticated probability tools have been generalized
and applied to random matrix theory (e.g., the Lindeberg Principle, by
Chatterjee~\cite{chatterjee} and Tao and Vu~\cite{taovu,taovu2}).

Most of the focus in universality research has been on proving, under
progressively weaker assumptions on the entry distribution, the following:
\begin{itemize}
\item[-] convergence of the ESD to the semicircle or Mar\v
{c}enko--Pastur laws and establishing rates of convergence in various
ways (large deviations, concentration estimates, central limit
theorems). For a comprehensive treatment of the subject, see the books
by Bai and Silverstein~\cite{baisil} and by Anderson, Guionnet and
Zeitouni~\cite{agzbook};
\item[-] fluctuations of the spectrum at the edge (the famous
Tracy--Widom laws~\cite{twl,twuniv1,twuniv2}) for general Wigner
matrices, settled by Tao and Vu~\cite{taovu2} and Erd\H{o}s et
al.~\cite{ERSTVY};
\item[-] universality of correlation functions in the bulk, under
various assumptions (Tao and Vu~\cite{taovu}, Erd\H{o}s et al. \cite
{ERSTVY}, Erd\H{o}s et al.~\cite{eprsy}, Erd\H{o}s, Yau and Yin
\cite{eyy});
\item[-] partial or complete delocalization of the eigenvectors (by
Erd\H{o}s, Schlein and Yau~\cite{esy1,esy2} and Tao and Vu
\cite{taovu,taovu2}).
\end{itemize}



The aim of said research has been to show that the spectral statistics
agree in the large $n$ limit to the spectral statistics of the Gaussian
Orthogonal Ensemble (GOE) and the Gaussian Unitary Ensemble (GUE),
depending on whether the matrices are real symmetric/positive-definite,
or complex Hermitian/positive-definite. The Gaussian and Wishart
ensembles are some of the most studied and best understood random
matrix models; for an easy introduction to classical random matrix
theory, see the books by Mehta~\cite{mehtabook} and Muirhead \cite
{muirhead82a}.

Parallel to these developments, in combinatorics and discrete
mathematics, there has always been an interest in studying spectral
properties of deterministic and random graphs. There are two matrices
of interest in spectral graph theory: the \textit{adjacency} (sometimes
called the \textit{incidence}) matrix, which we already defined, and the
\textit{Laplacian} matrix. These matrices are the same for regular
graphs (although, in general, they can be quite different), and the
spectrum of the graph is the spectrum of the matrix.




Among the properties of random graphs that have been the focus of
intense research are connectivity, phase transitions and the limiting
spectral distribution of random graphs, including trees (McKay
\cite{mckay81}, Feige and Ofek~\cite{feigeofek05}, Mirlin and Fyodorov
\cite{mirlinfyod91}, Bauer and Golinelli~\cite{bauergolinelli01},
Semerjian and Gugliandolo~\cite{semerjian02}, Bordenave and Lelarge
\cite{bordlelarge}, Bhamidi, Evans and Sen~\cite{bhamidievanssen}).
Other properties include concentration of eigenvalues (Krivelevich and
Sudakov~\cite{krivsudakov03}, Alon, Krivelevich and Vu \cite
{alonkrivvu}), the spectral gap (F\H{u}redi and Koml\'{o}s \cite
{furedikomlos}, Friedman~\cite{friedmane2} and Friedman and Alon \cite
{friedmanalon}, Broder and Shamir~\cite{brodershamir}).

Another area of recent interest is the study of quasi-random graphs and
expanders. These are nonrandom graphs which display properties one
expects to hold with high-probability for certain classes of random
graph models. For example, expanders are sparse graphs that have high
connectivity properties (e.g., a~large spectral gap). These graphs are
often regular (e.g., the famous Ramanujan graph, described in the
seminal articles by Lubotzky, Phillips and Sarnak~\cite{LPS} and
Morgenstern~\cite{morgenstern}). Random $d$-regular graphs display the
same connectivity properties with very high probability, when $d$ is
kept fixed and the order is large; this is in essence the Alon
conjecture, recently settled by Friedman~\cite{friedmanalon}. Thus a
study of random regular graphs suggests possible properties of
(deterministic) expanders.

It is easy for a probability audience to appreciate the importance of
studying eigenvalues of the graph (e.g., the spectral gap which
determines the mixing properties of a random walk), but eigenvectors of
graphs are equally important, especially since they are the solutions
of various combinatorial optimization problems. Traditionally, there
has been much less work on computing the \textit{actual} graph
eigenvector distributions,
with the notable and recent exception of Wishart-like sample covariance
matrices (see Bai, Miao and Pan~\cite{buymeapan}). Thus,
developments in examining properties of the eigenvectors of large
random graphs (as in Friedman~\cite{friedmannodal} and Dekel, Lee and
Linial~\cite{dekelleelinial}) are relatively new, and motivated by the
applications of eigenvectors to engineering and computer science. Such
applications include the Google page-rank algorithm~\cite{pagerank},
the Shi--Malik algorithm~\cite{shimalik00}, the Meila--Shi algorithm
\cite{meilashi} and other spectral clustering techniques and related
segmentation problems (Weiss~\cite{weiss99}, Pothen, Simon and Liou
\cite{pothensimonliou}, etc.).

It is probably clear by now that the two fields of research that we
have very briefly sketched here (universality studies in random matrix
theory and spectra of random graphs) are
vast and, by examining the two lists of important problems we have
outlined, one can see that there is a certain amount of overlap.
Naturally, this lead to a few papers where the two fields have
intersected, despite differences in both the goals and the methodology
of each.

A famous such example is McKay's derivation of the limiting empirical
spectrum of random $d$-regular graphs on $n$ vertices, as $d$ is fixed
and $n$ grows to infinity~\cite{mckay81}. In that case, the empirical
spectral distribution converges in probability to what is known as the
McKay (or Kesten--McKay) law, which has a density
%
%
\begin{equation}\label{whatismckayden}
f_d(x) = \frac{d\sqrt{4(d-1)-x^2}}{2\pi(d^2 - x^2)},\qquad -2\sqrt
{d-1} \le x \le2\sqrt{d-1}.
\end{equation}
This density had appeared earlier in Kesten's work on random walks
on
groups~\cite{kesten}. It can be easily verified that as $d$ grows to
infinity, if we normalize the variable $x$ in the above by $\sqrt
{d-1}$, the resulting density converges to the semicircle law on $[-2,2]$.

This naturally raised the question of whether the study of
``universal'' properties could be pushed into the domain of regular
random graphs with increasing degree. The rate of growth of the degree
sequence plays an important role, since at both extremes ($d$ fixed and
$d = n-1$) the ESDs do not converge to the semicircle law.

The answer to this question turns out to be difficult. There
are a number of major obstacles to developing an applicable
universality theory in the spirit of Wigner random matrices to
adjacency matrices of random graphs, which are non-Wigner: these
matrices are sparse and the entries are not independently distributed.

To see how sparsity affects concentration, consider the question of
proper scaling of the adjacency matrix $A_n$. In the Wigner case, the
scaling factor is clearly $1/\sqrt{n}$, which puts \textit{all} of the
eigenvalues in $[-2-\varepsilon,2+\varepsilon]$, for any positive
$\varepsilon$, with very high probability for a sufficiently large $n$.
One might be tempted then to believe that the proper scaling for
adjacency matrices is $1/\sqrt{d_n}$, as this achieves the same kind
of finite row-variance as $1/\sqrt{n}$ does in the Wigner case.
Unfortunately, it is not known if this scaling will place all the
eigenvalues (except the first) in a~compact interval, as $d_n
\rightarrow\infty$.

For the regime when $d$ is fixed, the Alon conjecture states that the
second largest eigenvalue (in absolute value) $\lambda_2$ has an upper
bound $|\lambda_2| \leq2 \sqrt{d-1}+\varepsilon$, with very high
probability. The well-known lower bound holds for every $d$-regular
graph and we cite it from Friedman~\cite{friedmane2}: $\vert\lambda _2
\vert \ge2\sqrt{d -1} + O( {\log d}/{\log n} )$. Unfortunately, when
$d$ grows with $n$, the upper bound is not known to hold outside of a
narrow growth regime.\setcounter{footnote}{2}\footnote{More precisely,
even under this narrow growth regime the theorem is valid only for
Friedman's permutation model. For $d$ fixed, Friedman's model
approaches the uniform distribution with increasing $n$; this is no
longer clear once $d_n$ grows with $n$.} Khorunzhy~\cite{khorunzhy01}
has shown that, for a random matrix model similar to the adjacency
matrix of the Erd\H{o}s--Ren\'{y}i random graph on $n$ vertices with an
expected degree $d_n \gg\log n$, with probability one, the spectral
norm of the adjacency matrix grows faster than $\sqrt{d_n}$. Although
this does not necessarily affect convergence of the ESD to the
semicircle law, it eliminates the possibility of containing all the
eigenvalues of the rescaled centralized adjacency matrix within any
compact interval.



Our results investigate the extent to which universality can be
extended to the slowly growing $d_n$ case. Our first result is Theorem
\ref{globalsemic} stated below.
\begin{theorem} \label{globalsemic}
Let $d_n$ satisfy the asymptotic condition
%
%
\begin{equation}\label{assumed1}
\lim_{n\rightarrow\infty}d_n = \infty,\qquad d_n-1= n^{\varepsilon
_n}\qquad \mbox{for some } \varepsilon_n=o(1).
\end{equation}
Then the ESD of the matrix $(d_n-1)^{-1/2} A_n$, where $A_n$ denotes
the adjacency matrix of $G_n$, converges in distribution to the
semicircle law on $[-2,2]$ which has a density
%
%
\begin{equation}\label{whatissscdensity}
\dsc(x) := \frac{1}{2\pi} \sqrt{4 - x^2},\qquad -2 < x < 2.
\end{equation}
\end{theorem}

The condition on $d_n$, for example, includes the logarithmic regime,
$d_n = (\log n)^{\gamma}$ for any positive $\gamma$; in which case we
can define $\varepsilon_n$ as $\gamma\log\log n / \log n$.

Our proof of this result (and the following ones)
depends crucially on two facts:
\begin{longlist}[(ii)]
\item[(i)] the ``locally tree-like'' property, which states that with
high probability, most vertices in a random regular graph will have a
(increasingly larger) neighborhood which is free of any cycles, and
\item[(ii)] the fact that $d_n$ grows to infinity, which smooths out
irregularities as~$n$ tends to infinity.\vadjust{\goodbreak}
\end{longlist}


Our second result is arguably the most important one in this paper.
\begin{theorem} \label{localsemic} Fix $\delta>0$. Let $d_n = (\log
n)^{\gamma}$, where $\gamma>0$. Let $\eta_n =(r_n - r_n^{-1})/2$
where $r_n = \exp( d_n^{-\alpha} )$ for some $0 < \alpha
< \min(1, 1/\gamma)$.
Then there exists an $N$ large enough such that for all $n \geq N$, for
any interval $I \subset\mathbb{R}$ of length $|I| \geq\max\{2\eta
_n, \eta_n / (-\delta\log\delta)\}$,
\[
\biggl\vert\ecount_I - n \int_I \dsc(x) \,dx \biggr\vert < \delta n \vert I
\vert
\]
with probability at least $1-o(1/n)$. Here $\ecount_I$ is the number
of eigenvalues of $\frac{1}{\sqrt{d_n-1}} A_n$ in the interval $I$,
and $\dsc$ refers\vspace*{1pt} to the density of the semicircle law as in~(\ref
{whatisfsc}).
\end{theorem}
\begin{rmk} Note that the shortest length of the interval $I$ that our
methods can narrow down to is of length $\eta_n$, which is roughly
about $1/\log n$, if $d_n \gg\log n$, and $1/d_n$, if $d_n \ll\log
n$. For Wigner matrices a far shorter scale can be achieved
(effectively poly-log over $n$ in~\cite{esy1}). Such sharp estimates
are not to be expected in the graph case, and this again is a
consequence of sparsity and lack of concentration estimates.
\end{rmk}
\begin{rmk} A close examination of the proof of Theorem \ref
{localsemic} reveals that it can be extended to any deterministic
sequence of regular graphs of increasing size and degree, as long as
the ``locally tree-like'' property holds at ``most'' vertices.
\end{rmk}
\begin{rmk} Since the submission of this paper, significant progress
has been made in proving the local semicircle law for random regular
graphs in \textit{any} kind of growth regime for $d_n$ (see Tran, Vu and
Wang~\cite{TVW}). Their methods rely on proving the local semicircle
law first for Erd\H{o}s--R\'{e}nyi graphs with suitable parameters,
and then using a result by McKay and Wormald~\cite{MW} about the
probability that an Erd\H{o}s--R\'{e}nyi graph is regular. Their
result subsumes ours (in the sense that the lower bound on the length
of the interval $I$ is smaller) for the case when $d_n = \Omega((\log
n)^{10})$; when $d_n = o((\log n)^{10})$, our result is slightly
stronger in the same sense.
\end{rmk}


Although these results are similar up to a point to the Wigner matrix
results, our methodology is essentially different. Due to sparsity and
lack of concentration, we had to adapt a more combinatorial set of
tools (in particular, the tree approximation) as well as tools from
linear algebra to the Stieltjes transform approach used in~\cite{esy1}
and~\cite{taovu}.

Several recent articles have done extensive simulations on eigenvalues
and eigenvectors of random graphs, with surprising conclusions. For
example, Jakobson et al.~\cite{simeval} carries out a numerical study
of fluctuations in the spectrum of regular graphs. Their experiments
indicate that the level spacing distribution of a generic $k$-regular
graph approaches that of the GOE as we increase the number of vertices.
On the eigenvector front, in the article~\cite{simevec} by Elon, the
author attempts to characterize the structure of the eigenvectors by
suggesting (with numerical observations) that all, except the first,
follow approximately a Gaussian distribution. Additionally, the local
covariance structure has been conjectured to be given by explicit
functions of the Chebyshev polynomials of the second kind. In
particular, if two vertices on the graph are at a distance $k$ from
each other, it is conjectured that the covariance between the
coordinates of any eigenvector at the two vertices decays exponentially
in $k$.

All this empirical data points to universality properties of the
adjacency matrices of large, sparse regular graphs; we took here a
first step toward proving them.

If the eigenvectors are indeed uniformly distributed over the sphere
then they must (with high probability) satisfy delocalization. We give
upper bounds on the probability of this phenomenon.

We use the following definition of delocalization, similar to the one
used in~\cite{esy1}.
\begin{defn}\label{deloc}
Let $T$ be a subset of $\{1,2,\ldots,n\}$ of size $L\ge1$. Let
$\delta> 0$ be some fixed number. We say that a vector $v=(v(1),
\ldots, v(n))\in\rr^n$ with~$\ltwo$ norm $\|v\|_2 = 1$ exhibits
$(T,\delta)$ localization if
\[
\|v\mid_T \|_2^2 = \sum_{j\in T} \vert v(j) \vert^2 \ge1- \delta.
\]
The vector $v$ is said to be $(L,\delta)$ localized there exists some
set $T \subset\{ 1, 2, \ldots, n\}$ such that $\vert T \vert=L$ such
that $v$ is $(T,\delta)$ localized.
\end{defn}

Below is our result on eigenvector delocalization.
\begin{theorem} \label{thmdeloc} Assume the set-up of Theorem \ref
{localsemic}. Fix $\delta>0$.

\begin{longlist}[(ii)]
\item[(i)] Let $T_n \subseteq\{1,2,\ldots,n\}$ be a deterministic
sequence of sets of size $L_n = o(\eta_n^{-1})$. Let $\Omega_1(n)$ be
the event that some normalized eigenvector of the matrix $A_n$ is
$(T_n,\delta)$ localized. Then, for all sufficiently large $n$,
\[
P( (\Omega_1(n))^c ) \ge e^{-L_n \eta_n / d_n} \biggl( 1
- o\biggl( \frac{1}{d_n} \biggr)\biggr) =
1 - o \biggl( \frac{1}{d_n} \biggr).
\]

\item[(ii)] Define the sequence
\[
\zeta_n= \frac{1}{4}\frac{\log n}{\log(d_n-1)} - 4,\qquad n\ge2.
\]
Consider the (random) subset $J(n)$ of all vertices in the graph whose
$\zeta_n$-neighbor\-hood is free of cycles. Then,
\[
P\biggl( \frac{\vert J(n) \vert}{n} \ge1- \frac{\eta_n}{d_n} \biggr) =
1- o\biggl( \frac{1}{n} \biggr).
\]
Moreover, there exists an $n$ large enough such that the event that
$T_n \subset J(n)$ and some normalized eigenvector is $(T_n,\delta)$
has probability zero.
\end{longlist}
\end{theorem}
\begin{rmk} More progress has been made on the eigenvector
delocalization front since the submission of this paper. In their paper
\cite{TVW}, the authors prove that the $\ell^{\infty}$ norms of all
eigenvectors are $o(1)$, regardless of the regime of growth of $d_n$.
Very recently, Erd\H{o}s et al. posted a paper~\cite{EKYY} proving
the local semicircle law for Erd\H{o}s--R\'{e}nyi graphs with $pn =
\Omega(\log n)$ up to a spectral window (an interval $I$) of size
larger than $1/n$; from this, they could deduce that the eigenvectors
of such Erd\H{o}s--R\'{e}nyi graphs are completely delocalized, that
is, that the $\ell^{\infty}$ norms of the normalized (unit)
eigenvectors are at most of order $1/\sqrt{N}$ with high probability.
It would be interesting to see if the methods of~\cite{TVW} (of
deducing results for random regular graphs from the same results for
Erd\H{o}s--R\'{e}nyi) can be combined with the theorems of \cite
{EKYY} to obtain complete eigenvector delocalization (and, potentially,
a~much smaller spectral window) for random regular graphs with $d_n =
\Omega(\log n)$.
\end{rmk}

%

The bounds in Theorem~\ref{thmdeloc} are not sharp. There are severe
technical obstacles in producing sharp bounds by adapting the strategy
of Wigner matrices. One such example is eigenvalue collision, that is,
the event that $A_n$ does not have $n$ distinct eigenvalues. Since
$A_n$ has discrete entries, this event has a~positive probability.
However, to the best of our knowledge, no good bound on this
probability is known.

The paper is organized as follows. In Section~\ref{global} we prove
the global convergence to the semicircle law (Theorem \ref
{globalsemic}). This is followed by the proof of the local semicircle
law (Theorem~\ref{localsemic}) in Section~\ref{local}. The
eigenvector delocalization is proved in Section~\ref{eigenvec}.
Finally, the \hyperref[app]{Appendix} contains an exact calculation of eigenvalues and
eigenvectors for the random regular (finite) tree, defined in Section
\ref{localsemic}.

\section{Global convergence to the semi-circle law}\label{global}

Recall that $G_n=G(n,d_n)$ denotes a random $d_n$-regular graph on $n$
vertices whose adjacency matrix is~$A_n$.
Recall that
$d_n$ satisfies the asymptotic condition
%
%
\begin{equation}\label{assumed}
\lim_{n\rightarrow\infty}d_n = \infty,\qquad d_n-1= n^{\varepsilon
_n}\qquad \mbox{for some } \varepsilon_n=o(1).
\end{equation}


We prove here Theorem~\ref{globalsemic}, namely, that the empirical
spectral distribution (ESD) of the adjacency matrix $A_n$ converges in
probability to the semicircle law on $[-2,2]$ which we recall from
(\ref{whatissscdensity}).

Our main instrument is to use the \textit{moment method}. Our
arguments depend crucially on the following local approximation of
$G_n$ by a rooted tree. Consider the deterministic rooted tree $S_n$,
which is the infinite regular tree of degree $d_n$ with a distinguished
vertex marked as the \textit{root}. For a graph $G$ whose every edge
is taken to have unit length, consider the induced metric structure on
$G$. We define the $r$-neighborhood of the vertex $i$, to be the
subgraph of $G$ whose vertices are at a distance at most $r$ from $i$,
and whose edges are all the edges between those vertices. The following
lemma makes precise the idea that, except for a vanishing proportion of
the vertices, the $r$-neighborhood of any vertex is isomorphic to the
corresponding neighborhood of the root in the tree $S_n$.

Recall that a cycle is a sequence of vertices $\{i_1, \ldots, i_k\}$
of a graph such that \mbox{$i_1=i_k$}, there is no other repeated vertex, and
there is an edge between every successive $i_j$ and $i_{j+1}$. The
length of the cycle is the number of vertices except the initial one. A
cycle of length $k$ will be called a $k$-cycle. Finally, a~cycle-free
or \textit{acyclic} graph is a tree.
\begin{lemma}\label{treeapp1}
Fix a positive integer $r$. Let $\tau(n)$ be the subset of vertices
of~$G_n$ which have no cycles in their $r$-neighborhoods, and let $\vert
\tau(n) \vert$ denote the size of~$\tau(n)$. Then, under the assumptions
of (\ref{assumed}), we have
\[
P\biggl( 1 - \frac{\vert\tau(n) \vert}{n} > n^{-1/4} \biggr) =o(
n^{-5/4} ).
\]
\end{lemma}
\begin{pf}
We use the estimates of McKay, Wormald and Wysocka~\cite{mww04} on the
Poisson approximation to the number of short cycles in regular graphs.
Let $g(n)$ be a sequence such that
%
%
\begin{equation}\label{whatisgn}
g(n) \ge3\quad \mbox{and}\quad (d_n-1)^{2g(n)-1}=o(n).
\end{equation}
For any $s\le g$, let $M_s$ denote the number of cycles of length $s$
in the graph~$G_n$. It has been shown in~\cite{mww04} that $M_s$ is
approximately distributed as a~Poisson random variable and
%
%
\begin{eqnarray}\label{evarm}
E( M_s )&=& \mu_s\bigl(1+ O\bigl( s(s+d)/n \bigr)\bigr)\qquad
\mbox{where } \mu_s = \frac{(d-1)^s}{2s}\quad\mbox{and}\nonumber\\[-8pt]\\[-8pt]
\operatorname{Var}( M_s )&=&\mu_s +
O\bigl( s(s+d)/n \bigr)\mu_s^2.
\nonumber
\end{eqnarray}
Consider now the growth of the degree sequence as in (\ref{assumed}).
If we choose $g(n)$ such that $2g(n)-1=1/\sqrt{\varepsilon_n}$, it will satisfy
\[
(d_n -1)^{2g(n) -1} = n^{\sqrt{\varepsilon_n}}= o(n).
\]
Additionally $g(n)$ grows to infinity with $n$, since $\varepsilon(n)=o(1)$.

Consider an $s$-cycle for some $s\le g(n)$. It has exactly
$s$-vertices. Now, the number of vertices whose $r$ neighborhoods fail
to be acyclic because of this $s$-cycle are precisely those vertices
which are at a distance of at most $(2r-s)/2$ from any of the vertices
in the $s$-cycle. The number of such vertices has an easy upper bound
of $2(d_n-1)^{(2r-s)/2}s$, for all large enough $d_n$. Thus, the total
number of vertices whose $r$ neighborhoods are not acyclic can be
bounded above by
%
%
\begin{equation}\label{whatisnr}
N^*_r=\sum_{s=3}^{2r}2 s (d_n-1)^{(2r-s)/2} M_s.
\end{equation}
Also,
%
%
\begin{equation}\label{taunn}
n - \vert\tau(n) \vert \le N_r^*.
\end{equation}

Taking expectations on both sides of (\ref{whatisnr}), and using
formulas (\ref{evarm}), we get
\[
E N^*_r = \sum_{s=3}^{2r} 2s (d-1)^{(2r-s)/2} \frac
{(d-1)^s}{2s}\bigl( 1 + O\bigl( s(s+d)/n \bigr) \bigr).
\]

The quantity $O(s(s+d)/n)$ denotes a function $h(s,d,n)$ such that
\[
\frac{n}{s(s+d)} h(s,d,n)
\]
remains bounded for all choices of $s,d$ and $n$. Thus we get
\[
E N^*_r= (d-1)^{r} \sum_{s=3}^{2r} (d-1)^{s/2} + O\Biggl( \frac
{1}{n}\sum_{s=3}^{2r} s(s+d)(d-1)^{r+s/2}\Biggr)=O\bigl( (d-1)^{2r}
\bigr).
\]
The last equality is true since, by our assumption on $d_n$, the second
term in the sum is $o(1)$.

Similarly, we can compute the second moment. By the Cauchy--Schwarz inequality,
\begin{eqnarray*}
\var(N^*_r)&\le& 2r \sum_{s=3}^{2r} 4s^2(d-1)^{2r-s} \var(M_s)\\
&\le& 2r \sum_{s=3}^{2r} 4s^2(d-1)^{2r-s}\bigl[ \mu_s + O\bigl(
s(s+d)/n \bigr)\mu_s^2\bigr]\\
&\le& 2r \sum_{s=3}^{2r} 4s^2(d-1)^{2r-s} \mu_s + 2r \sum_{s=3}^{2r}
4s^2(d-1)^{2r-s} O\bigl( s(s+d)/n \bigr) \mu_s^2.
\end{eqnarray*}
Plugging in the value of $\mu_s$ from (\ref{evarm}) we get
\[
\var(N^*_r)\le4r^2(2r+1) (d-1)^{2r} + 2r (d-1)^{2r}\sum_{s=3}^{2r}
(d-1)^{s} O\bigl( s(s+d)/n \bigr).
\]
As before, it thus follows that
\begin{eqnarray*}
\sum_{s=3}^{2r}(d-1)^{s} O\bigl( s(s+d)/n \bigr)
&=& O\Biggl( \sum
_{s=3}^{2r} (d-1)^{s} s(s+d)/n \Biggr)\\
&=&O\Biggl( n^{-1}(2r+d)(d-1)^{2r} \sum_{s=3}^{2r} s\Biggr)\\
&=& O\bigl(
n^{-1}(2r+d)(d-1)^{2r} {r}(2r+1) \bigr).
\end{eqnarray*}
Hence
\[
\var(N^*_r)\le4 r^2(2r+1) (d-1)^{2r} + O\bigl( r^2 (2r+d)
(d-1)^{4r}/n \bigr).
\]
Note again that, by our assumption, the quantity $(d-1)^{4r}/n$ is $o(1)$.

We now want to use Markov's inequality to bound the tail probability of
the quantity $1-\vert\tau(n) \vert/n$. Fix any $\varepsilon> 0$. Then, by
inequality (\ref{taunn}), we get
\begin{eqnarray*}
P\biggl( 1 - \frac{\vert\tau(n) \vert}{n} > \varepsilon\biggr)&\le& P
( N^*_r > n\varepsilon)\le\frac{1}{n^2\varepsilon^2} E(
N^*_r )^2\\
&=& \frac{1}{n^2\varepsilon^2}[ \var(N^*_r) + (E(
N^*_r ))^2 ]\\
&\le&\frac{1}{n^2\varepsilon^2 } \bigl[ 4 r^2(2r+1) (d-1)^{2r}\\
&&\hphantom{\frac{1}{n^2\varepsilon^2}[}{} + r^2
(2r+d) o(1) + O\bigl( (d-1)^{4r} \bigr) \bigr]\\
&\le&\varepsilon^{-2} O\biggl( \frac{(d-1)^{4r}}{n^2} \biggr)\\
&=& \varepsilon
^{-2}O( n^{4r\varepsilon_n -2} )
\end{eqnarray*}
by our choice of the sequence $d_n$.

Choosing $\varepsilon=n^{-1/4}$ we get
\[
P\biggl( 1 - \frac{\vert\tau(n) \vert}{n} > {n^{-1/4}}\biggr)\le\sqrt
{n}O( n^{4r\varepsilon_n -2} )=o( n^{-5/4} ),
\]
since $\varepsilon_n=o(1)$. This completes the proof of the lemma.
\end{pf}
\begin{lemma}\label{wkconvlem} Let $\{ \mu_i, i=1,2,\ldots\}$ be
a sequence of random probability measures on the real line, defined on
the same probability space. Let $\mu$ be a~nonrandom continuous
probability measure supported on a compact interval $I$. Suppose there
exits a pair of doubly indexed real-valued sequences $\{ a_n(r),
b_n(r),\break r,n \in\mathbb{N}\}$ such that the following hold:
\begin{longlist}[(1)]
\item[(1)] For every $r=1,2,\ldots,$ we have
\[
P\Biggl(\bigcup_{N=1}^\infty\bigcap_{n \ge N}
\biggl\{ a_n(r) \le\int x^r \,d\mu_n(x) \le b_n(r) \biggr\} \Biggr) =1.\vadjust{\goodbreak}
\]
\item[(2)] For every $r=1,2,\ldots,$ we have
\[
\lim_{n\rightarrow\infty} a_n(r) = \lim_{n\rightarrow\infty}
b_n(r) = \int x^r \,d\mu(x) < \infty.
\]
\end{longlist}
Then the sequence of measures $\{ \mu_n\}$ converges to $\mu$ in probability.
\end{lemma}
\begin{pf}
Let $\Omega_r$ be the event
\[
\bigcup_{N=1}^\infty\bigcap_{n \ge N} \biggl\{
a_n(r) \le\int x^r \,d\mu_n(x) \le b_n(r) \biggr\}.
\]
Then, from condition (1), it follows that
\[
1-P\Biggl( \bigcap_{r=1}^\infty\Omega_r\Biggr) = P\Biggl(
\bigcup_{r=1}^\infty\Omega_r^c \Biggr) \le\sum
_{r=1}^\infty P( \Omega_r^c) =0.
\]
Thus $P( \bigcap_{r=1}^\infty\Omega_r)=1$.

Consider any fixed realization of the sequence $\{\mu_n\} \in\bigcap
_{r=1}^\infty\Omega_r$. By Helly's selection theorem, this sequence
has a limit point $\nu$. Thus, there is a subsequence $\{ \mu_{n_k} \}
$ that converges to $\nu$ in the topology of weak convergence.

Now take $r$ to be a positive integer. We would like to show that
\[
\lim_{n_k \rightarrow\infty} \int x^r \,d\mu_{n_k} = \int x^r \,d\nu.
\]
From the standard theory of weak convergence, it follows that this will
be true if the function $x^r$ is uniformly integrable under the
sequence of measures $\{\mu_{n_k}\}$. However, uniform integrability
follows from the following $L^2$-boundedness condition:
\[
\max_{n_k} \int x^{2r} \,d\mu_{n_k} < \max_{n_k} b_{n_k}(2r) < \infty
\]
by conditions (1) and (2).

In particular, from condition (2) we reach the conclusion
\[
\int x^r \,d \nu(x) =\int x^r \,d\mu(x),\qquad r=0,1,2,\ldots.
\]
Since the support of $\mu$ is the compact interval $I$, it follows
that the moment problem has a unique solution, and hence, $\nu$ must
be equal to $\mu$.

This shows that any limit point of any sequence $\{\mu_n\}$ in $\bigcap
_{r=1}^\infty\Omega_r$ is given by $\mu$. By the usual subsequence
argument, this shows that $\mu_n$ converges to~$\mu$ in the set $\bigcap
_{r=1}^\infty\Omega_r$, and hence with probability one. This proves
the result.
\end{pf}
\begin{pf*}{Proof of Theorem~\ref{globalsemic}}
Consider the random graph sequence $G_n=G(n,d_n)$ as in the statement,
and let $A_n$ be the adjacency matrix of $G_n$.


Let $\mu_n$ be the ESD of the matrix $(d_n - 1)^{-1/2} A_n$. Then, for
any positive integer~$r$,
\[
\int x^r \,d\mu_n(x) = \frac{1}{n} \tr\bigl( (d_n -1)^{-r/2} A_n^r
\bigr)= \frac{(d_n -1)^{-r/2} }{n} \sum_{i=1}^n A^r_n(i,i).
\]
Here $A_n^r(i,i)$ is the $i$th diagonal element of the matrix $A_n^r$.

Note that $A_n^r(i,i)$ counts the number of paths of length $r$ that
start and end at $i$. Consider the set of vertices in $\tau(n)$, as in
Lemma~\ref{treeapp1}, whose $\lceil r/2\rceil$-neighborhood is
acyclic. For any $i\in\tau(n)$, the number of such paths, $B_n^r$, is
equal to the number of paths of size $r$ that start and end at the root
of the tree $S_n$. If $i\notin\tau(n)$, we use the trivial bound
$A^r_n(i,i) \le d_n^{r}$.
Thus
\begin{eqnarray*}
(d_n -1)^{-r/2}\frac{\vert\tau(n) \vert}{n} B_n^r &\le&\frac{(d_n
-1)^{-r/2} }{n} \sum_{i=1}^n A^r_n(i,i)\\
&\le& (d_n -1)^{-r/2}\biggl[B_n^r + \frac{n - \vert\tau(n) \vert}{n}
d_n^{r}\biggr].
\end{eqnarray*}
If we define
%
%
\begin{eqnarray}\label{whatisanbnr}
a_n(r)&=&( 1- n^{-1/4})(d_n -1)^{-r/2}B_n^r,\nonumber\\[-8pt]\\[-8pt]
b_n(r)&=&(d_n -1)^{-r/2}[B_n^r + n^{-1/4}d_n^{r/2}],\nonumber
\end{eqnarray}
then from Lemma~\ref{treeapp1} we get
%
%
\begin{equation}\label{finitesum}
P\biggl( a_n(r)\le\int x^r \,d\mu_n(x) \le b_n(r)\biggr) \ge1- o
(n^{-5/4}).
\end{equation}
In particular, by taking complements of the events above, we get
\[
\sum_{n=1}^\infty P\biggl( \biggl\{ a_n(r)\le\int x^r \,d\mu_n(x) \le
b_n(r) \biggr\}^c\biggr) < \sum_{n=1}^{\infty} o
(n^{-5/4})< \infty.
\]

Now consider a product probability space on which independent copies of
our (countably many) random graphs are defined. Applying the
Borel--Cantelli lemma and (\ref{finitesum}) we get that
\[
P\Biggl( \bigcap_{N=1}^\infty\bigcup_{n=N}^\infty
\biggl\{ a_n(r)\le\int x^r \,d\mu_n(x) \le b_n(r) \biggr\}^c \Biggr)=0.
\]
Taking the complements again, we get
\[
P\Biggl( \bigcup_{N=1}^\infty\bigcap_{n=N}^\infty
\biggl\{ a_n(r)\le\int x^r \,d\mu_n(x) \le b_n(r) \biggr\} \Biggr)=1.
\]
This satisfies condition (1) in Lemma~\ref{wkconvlem}.

Once we show the validity of condition (2) for $\mu$ equal to the
semicircle law, we will be done by Lemma~\ref{wkconvlem}. Clearly, by
our choice of $d,n$ as in the statement,\vadjust{\goodbreak} and the functions $a_n(r),
b_n(r)$ as in (\ref{whatisanbnr}), this will be true once we establish
\[
\lim_{n \rightarrow\infty}(d_n -1)^{-r/2}B_n^r= \int x^r \dsc(x) \,dx.
\]
We only need to verify above for even $r$, since for odd $r$, both
sides are zero ($B_n^r=0$ since in a tree one cannot return to the root
in an odd number of steps, and the moment is zero since $\dsc$ is a
symmetric density).

Now, for an even $r$, the value of $B_n(r)$ has been computed by McKay
in~\cite{mckay81} (denoted by $\theta(r)$ in equation (15) on \cite
{mckay81}). It is given by
\[
B_n(r)= \int_{-2\sqrt{d_n-1}}^{2\sqrt{d_n-1}} x^r f_n(x) \,dx,
\]
where $f_n(x)$ is the Kesten--McKay density
\[
f_n(x)= \frac{d_n\sqrt{4(d_n-1)-x^2}}{2\pi(d_n^2 - x^2)},\qquad
-2\sqrt{d_n-1} < x < 2\sqrt{d_n-1}.
\]

Thus, changing variable to $y=(d_n-1)^{-1/2}x$, we get
\begin{eqnarray*}
&&\lim_{n \rightarrow\infty}(d_n -1)^{-r/2}B_n^r \\
&&\qquad= \lim
_{n\rightarrow\infty}\frac{1}{(d_n-1)^{r/2}}\int_{-2\sqrt
{d_n-1}}^{2\sqrt{d_n-1}} x^r\frac{d_n\sqrt{4(d_n-1)-x^2}}{2\pi
(d_n^2 - x^2)}\,dx\\
&&\qquad=\lim_{n\rightarrow\infty}\int_{-2}^2 y^r \frac{d_n\sqrt
{d_n-1}\sqrt{4-y^2}}{2\pi(d_n-1)^2( (d_n/(d_n-1))^2 - y^2/(d_n-1) )}
(d_n-1)^{1/2}\,dy\\
&&\qquad=\frac{1}{2\pi} \int_{-2}^2 y^r \dsc(y) \,dy.
\end{eqnarray*}
The last equality follows by the dominated convergence theorem and the
fact that $\lim_{n\to\infty} d_n=\infty$. This completes our proof.
\end{pf*}

\section{Estimating the rate of convergence of the ESD}\label{local}

This is the longest section of the paper, and it is quite technical, so
we provide an outline of the proof. The approach we will use is given
by the Stieltjes transform of the adjacency matrix of the graph. To
estimate how far the Stieltjes transform of the graph is from the
Stieltjes transform of the semicircle, we will use as a~stepping stone
the resolvent of the adjacency matrix of a finite regular tree, which
we will show to be very close to both.

The estimation consists of the following steps:
\begin{longlist}[Step 3.]
\item[Step 0.] Basic definitions and properties of the quantities
involved (Section~\ref{local0}).
\item[Step 1.] Compute the resolvent of the regular tree, and show
that, in a~certain growth regime for $d_n$, its (root, root) elements
is close to the Stieltjes transform of the semicircle (Section~\ref{local1}).\vadjust{\goodbreak}
\item[Step 2.] Show that, in the same growth regime as before, the
(root, root) element of the resolvent of the regular tree is very close
to the Stieltjes transform of the regular graph (Section~\ref{local2}).
\item[Step 3.] Use the estimations from the previous steps to conclude
that the Stieltjes transform of the regular graph is close to that of
the semicircle, and use the methods of~\cite{taovu} to obtain bounds
on the rate of convergence of the ESD (Section~\ref{local3}).
\end{longlist}

\subsection{Basic definitions} \label{local0}\vspace*{-3pt}

%
\begin{defn}
For a $n\times n$ Hermitian matrix $A$ and a variable $z \in\cp$ for
which $\im(z)>0$ (thus $z$ is not an eigenvalue of $A$), define the
Stieltjes transform to be the function
\begin{eqnarray*}
s(A;z) & := &\frac{1}{n}\tr( A - zI_n )^{-1} \\
& := & \frac{1}{n}\tr( A - z )^{-1}.
\end{eqnarray*}
Here $I_n$ is the $n \times n$ identity matrix; for convenience, we
will drop the identity matrix and use the second notation.\vspace*{-3pt}
\end{defn}

We will also require the notion of Chebyshev orthogonal polynomials of
a complex variable. For more details, see the book by Mason and
Handscomb~\cite{MH}, page~14.

For a complex number $z$, define
\[
w= z+ \sqrt{z^2-1},\qquad z=\tfrac{1}{2}( w + w^{-1} ),
\]
%
where the square root of a complex number is taken such that the
imaginary part is always positive. It can be verified easily that for
any $r > 1$, the set
%
%
\begin{equation}\label{whatiser}
E_r:=\{ z\dvtx \vert w \vert= r \}
\end{equation}
is an ellipse whose foci are at $\{+1, -1\}$.

Note that
\[
\max_{z \in E_r} |\Im(z)| = \frac{r - {1/r}}{2},
\]
and that when $r=1$, this ellipse degenerates to the interval $[-1,1]$.\vspace*{-3pt}
\begin{defn}
The $n$th Chebyshev polynomial of the second kind $U_n$ is defined as
%
%
\begin{equation}\label{whatisun}
U_n(z)= \frac{w^{n+1} - w^{-(n+1)}}{w - w^{-1}},\qquad n=1,2,\ldots,
\end{equation}
with $U_0(z)\equiv1$. It is easy to check that, in addition, $U_n(z)$
satisfies the recursion
%
%
\begin{equation}\label{chebyrec}
U_n(z) = 2z U_{n-1}(z) - U_{n-2}(z),\qquad n=1,2,\ldots,
\end{equation}
with the initial conditions $U_0(z)=1, U_{-1}(z)=0$.\vadjust{\goodbreak}
\end{defn}
\begin{rmk}
When $r=1$, the above gives us the traditional orthogonal polynomials
for the semicircle law on the interval $[-1,1]$.
\end{rmk}

We will need the following bound on $U_n$ which can be found in
\cite{MH}, equations (1.53), (1.55):
%
%
\begin{equation}\label{chebyest}
\frac{r^n - r^{-n}}{r+ r^{-1}} \le\vert U_{n-1}(z) \vert \le\frac
{r^n -
r^{-n}}{r-r^{-1}},\qquad z \in E_r.
\end{equation}

Finally, we will need the standard formula for inverses of symmetric
block matrices, given below.
\begin{prop} Let $\mathbb{A}$ and $\mathbb{D}$ be complex symmetric
matrices with sizes \mbox{$n \times n$}, respectively, $m \times m$, and let
$\mathbb{B}$ be an $m \times n$ real matrix. Define the $(m+n) \times
(m+n)$ complex symmetric matrix
\[
\mathbb{M}=
\left[\matrix{
\mathbb{A} & \mathbb{B}\cr
\mathbb{B}' & \mathbb{D}}
\right]
,
\]
where $\mathbb{B}'$ denotes the transpose of $\mathbb{B}$%
. Then
%
%
\begin{equation}\label{invform1}
\mathbb{M}^{-1}=
\left[\matrix{
\mathbb{A}^{-1} + \mathbb{A}^{-1}\mathbb{B}\mathbb{F}^{-1}\mathbb
{B}'\mathbb{A}^{-1} & - \mathbb{A}^{-1}\mathbb{B}\mathbb{F}^{-1}\cr
- \mathbb{F}^{-1}\mathbb{B}' \mathbb{A}^{-1} & \mathbb{F}^{-1}}\right]
,\qquad \mathbb{F}= \mathbb{D} - \mathbb{B}' \mathbb{A}^{-1}
\mathbb{B}.\hspace*{-28pt}
\end{equation}
Equivalently, by reversing the roles of the blocks $\mathbb{A}$ and
$\mathbb{D}$,
%
%
\begin{eqnarray}\label{invform2}
\mathbb{M}^{-1}=
\left[\matrix{
\mathbb{G}^{-1}& - \mathbb{G}^{-1}\mathbb{B} \mathbb{D}^{-1} \cr
- \mathbb{D}^{-1}\mathbb{B}'\mathbb{G}^{-1} & \mathbb{D}^{-1} +
\mathbb{D}^{-1}\mathbb{B}'\mathbb{G}^{-1}\mathbb{B}\mathbb{D}^{-1}}\right]
,\nonumber\\[-8pt]\\[-8pt]
&&\eqntext{\mathbb{G}= \mathbb{A} - \mathbb{B}\mathbb{D}^{-1}
\mathbb{B}'.}
\end{eqnarray}
\end{prop}

These formulas are easy to verify, and their proofs can be found in
standard matrix algebra books.

\subsection{Resolvents of regular and almost regular trees} \label{local1}


Fix a positive integer \mbox{$d \ge2$}. Let $\tree$ be a finite ordered
rooted tree of depth $\zeta\in\mathbb{N}$ such that every vertex has
exactly $(d-1)$ \textit{children}. That is, the root has degree
$(d-1)$, and every other vertex, except the leaves, has degree $d$.
Such a tree is almost regular since all vertices, excluding the root
and the leaves, have degree $d$.

In order to define the adjacency matrix of this graph, we must fix
a~labeling; we will define this labeling recursively down to $\zeta=0$,
in which case all we have is a root vertex which we label $1$.

Imagine the tree embedded in the plane. If the depth is zero, the only
element is the root, and the adjacency matrix is obvious. If the depth
is one, the root has \mbox{$(d-1)$} children. Consider each child vertex as a
tree of depth zero, order their adjacency matrices $H_1, H_2, \ldots
, H_{d-1}$ from left to right, and consider a block matrix with these
as the diagonal blocks from upper left to bottom right. Finally add a
bottom-most row and a rightmost column for the root vertex.

By induction, suppose we have labeled the adjacency matrix for the tree
of depth $\zeta-1$. Consider now the tree of depth $\zeta$. If we
remove the root and the edges incident to it, we are left with $(d-1)$
trees of depth $\zeta-1$ arranged from left to right. We consider
their $(d-1)$ adjacency matrices and arrange them as diagonal blocks
and add the root as the last element.

Denote by $H$ the adjacency matrix thus obtained.
\begin{lemma}\label{treeform}
For any complex number $z$ such that $\Im(z) > 0$, and recall the
$n$th order Chebyshev polynomial, $U_n(z)$. Then the elements of the
resolvent of the adjacency matrix $H$, $(1/\sqrt{d-1} H-z)^{-1}$, have
the following properties:
\begin{longlist}[(iii)]
\item[(i)] 
\begin{eqnarray*}
\\[-32pt]
\biggl( \frac{1}{\sqrt{d-1}}H-z \biggr)^{-1}_{\Root, \Root}&=&\frac
{-1}{z+}\frac{-1}{z+}\cdots\frac{-1}{z},
\end{eqnarray*}
where the previous refers to a continued fraction of depth $\zeta$
(i.e., $\zeta+1$ recursions);
\item[(ii)] 
the above can also be represented as
%
%
\begin{equation}\label{rootroot}
\varphi(\zeta)=\biggl( \frac{1}{\sqrt{d-1}}H-z \biggr)^{-1}_{\Root
, \Root}=-\frac{U_{\zeta}(z/2)}{U_{\zeta+1}(z/2)};
\end{equation}
\item[(iii)] furthermore,
%
%
\begin{equation}\label{rootleaf}
\psi(\zeta)=\biggl( \frac{1}{\sqrt{d-1}}H-z \biggr)^{-1}_{\Root,
\leaf}=- \frac{(d-1)^{-\zeta/2}}{U_{\zeta+1}(z/2)},
\end{equation}
where $\leaf$ represents any leaf of $\mathbb{T}$.
\end{longlist}
\end{lemma}
\begin{pf}
%
(i) Note that when $\zeta=0$ (i.e., the tree has only the root
vertex), the equality is trivially true. We proceed by induction.
Suppose the equality is true until depth $\zeta-1$. Consider a tree of
depth $\zeta$ and label the adjacency~$H$ matrix as above. Thus
%
%
\begin{eqnarray}\label{blockrep}
&&\frac{1}{\sqrt{d-1}}H-z\nonumber\\[-4pt]\\[-12pt]
&&\qquad=
{\fontsize{10.7pt}{11pt}\selectfont{\left[\matrix{
\dfrac{1}{\sqrt{d-1}}H_1-z & & & & \vspace*{2pt}\cr
& \dfrac{1}{\sqrt{d-1}}H_2-z & & & u \vspace*{2pt}\cr
& & \cdots& &\vspace*{2pt}\cr
& & & \dfrac{1}{\sqrt{d-1}}H_{d-1}-z &\vspace*{2pt}\cr
&& u' && -z}\right]}}.\hspace*{-26pt}\nonumber
\end{eqnarray}
Here $u$ is the column vector representing the children of the root.
Notice that $u$ is $(d-1)^{-1/2}$ exactly at the $(d-1)$ coordinates
which are the last elements in each of the block matrices $H_1, \ldots
, H_{d-1}$ and zero elsewhere. The vector $u'$ is the transpose of $u$.

We now use formula (\ref{invform1}) treating the the final element
$[-z]$ as one block. Thus if $\varphi(\zeta)$ denote the element on
the left-hand side of (i) above, we get $\varphi(\zeta)=F^{-1}$, where
%
%
\begin{equation}\label{whatisf}
F= -z - \frac{1}{d-1}\sum_{i\sim\Root} \varphi(\zeta-1)= - z -
\varphi(\zeta-1).
\end{equation}
Here ``$i \sim\Root$'' refers to the children of the root which are,
in their turn, the roots of trees of depth $\zeta-1$. The formula now
follows by induction.

(ii) We will use the three term recurrence formula for
continued fractions which we state below. More details can be found in
the excellent book by Lorentzen and Waadeland~\cite{LW}, pages 5 and 6.
Given sequences of complex numbers $\{ a_n\}$ and $\{ b_n\}$ and a
complex argument $\omega$, one can define a continued fraction
function with argument $\omega$ by defining
\[
S_n(\omega) = b_0 + \frac{a_1}{b_1+} \frac{a_2}{b_2+} \cdots\frac
{a_n}{b_n + \omega}.
\]
By Lemma 1.1 in~\cite{LW} we get the existence of complex sequences
$\{ A_n\}$ and~$\{B_n\}$ such that
\[
S_n(\omega)= \frac{A_{n-1}\omega+ A_n }{B_{n-1}\omega+ B_n}
\qquad\mbox{for } n=1,2,\ldots,
\]
where
%
%
\begin{equation}\label{3term}
A_n = b_n A_{n-1} + a_n A_{n-2},\qquad B_n = b_n B_{n-1} + a_n B_{n-2}
\end{equation}
with initial values $A_{-1}=1, A_0=b_0, B_{-1}=0$ and $B_0=1$.

In our case we will take each $a_i=-1$ and each $b_i=z$, except
$b_0=0$. The recursions in (\ref{3term}) give us
\[
A_n = z A_{n-1} - A_{n-2},\qquad B_n = z B_{n-1} - B_{n-2}
\]
with the initial values $A_{-1}=1, A_0=0, B_{-1}=0$ and $B_0=1$.

Comparing with the recursions of the Chebyshev polynomials $U_n$ given
in (\ref{chebyrec}) we get that
\[
B_n(z)= U_n( z/2 ),\qquad A_n(z)= -U_{n-1}(z/2).
\]
Since clearly $\varphi(\zeta)=S_{\zeta+1}(0)$ we get formula (\ref
{rootroot}). This proves part (ii).

(iii) Since there is an obvious isomorphism of the tree that
can exchange the labeling of leaves, it is enough to consider the leaf
labeled $1$ in the adjacency matrix.
We will express
\[
\psi(\zeta)= \biggl( \frac{1}{\sqrt{d-1}} H - z
\biggr)^{-1}_{1,\Root}
\]
in terms of $\psi(\zeta-1)$.\vadjust{\goodbreak}

Let $N$ be the total number of vertices in the tree, and let $\mathcal
{A}$ be the diagonal block matrix which is the upper left $(N-1) \times
(N-1)$ block in (\ref{blockrep}).
Then from the formula of inverses of block matrices we get
\[
\psi(\zeta)= - F^{-1} (\mathcal{A}^{-1}u)_1,
\]
where\vspace*{1pt} $F$ is defined in (\ref{whatisf}). But $F^{-1}$ is $\varphi
(\zeta)$, and simplifying the elements of $w = \mathcal{A}^{-1}u$, we
see that $w(1) = (d-1)^{-1/2}\psi(\zeta-1)$. In other words, $\psi
(\zeta)=- (d-1)^{-1/2}\varphi(\zeta)\psi(\zeta-1)$.
We get by induction
\[
\psi(\zeta) = \frac{(-1)^{\zeta}}{(d-1)^{\zeta/2}} \prod
_{i=0}^{\zeta} \varphi(i),\qquad \zeta=1,2,\ldots.
\]
Now we substitute formula (\ref{rootroot}) to obtain (\ref{rootleaf}).
\begin{eqnarray*}
\psi(\zeta) &=& \frac{(-1)^{\zeta}}{(d-1)^{\zeta/2}} \prod
_{i=0}^{\zeta} - \frac{U_{i}(z/2)}{U_{i+1}(z/2)}= - \frac
{1}{(d-1)^{\zeta/2}}\frac{U_0(z/2)}{U_{\zeta+1}(z/2)}\\
&=&- \frac
{(d-1)^{-\zeta/2}}{U_{\zeta+1}(z/2)}.
\end{eqnarray*}
\upqed\end{pf}

Having now calculated the quantities $\varphi$ and $\psi$ for this
``slightly irregular'' tree~$\tree$, let us use them to find the
corresponding quantities for the regular one, where the root is
adjacent (just like all of the other nonleaf nodes) to precisely $d$
edges (and thus has $d$ children). We consider the same kind of
labeling as before.

Lemma~\ref{lemmatreeform2} below is a variation of Lemma \ref
{treeform} above.
\begin{lemma}\label{lemmatreeform2}
Let $\tree_d$ denote a $d$-regular tree of depth $\zeta\in\mathbb
{N}$ such that every vertex has degree $d$. Let $H_d$ denote the
adjacency matrix of the graph. The entries of its resolvent $(1/\sqrt
{d-1} H_d-z)^{-1}$ have the following properties:
\begin{longlist}[(ii)]
\item[(i)]
%
%
\begin{eqnarray}\label{drootroot}
\varphi_d(\zeta):\!&=&\biggl( \frac{1}{\sqrt{d-1}}H_d-z
\biggr)^{-1}_{\Root, \Root}\nonumber\\[-8pt]\\[-8pt]
&=& - \frac{U_{\zeta}(z/2)}{U_{\zeta+1}(z/2) -
(d-1)^{-1} U_{\zeta-1}(z/2)};\nonumber
\end{eqnarray}
\item[(ii)] the above can also be represented as
%
%
\begin{eqnarray}\label{drootleaf}
\psi_d(\zeta):\!&=&\biggl( \frac{1}{\sqrt{d-1}}H_d-z
\biggr)^{-1}_{\Root, \leaf}\nonumber\\[-8pt]\\[-8pt]
&=&- \frac{(d-1)^{-\zeta/2}}{U_{\zeta+1}(z/2) -
(d-1)^{-1} U_{\zeta-1}(z/2)},\nonumber
\end{eqnarray}
where $\leaf$ represents any leaf of $\mathbb{T}_d$.\vadjust{\goodbreak}

%
\end{longlist}
\end{lemma}
\begin{pf}
The proof is identical to that of the last lemma, except that we need
to be careful in the first step of the recursion. Since the labeling of
the vertices has the same principle as before, we have
%
%
\begin{equation}\label{dfromnond}
\biggl( \frac{1}{\sqrt{d-1}}H_d-z \biggr)^{-1}_{\Root, \Root}=
\frac{1}{- z - ({d}/({d-1})) \varphi(\zeta-1)},
\end{equation}
where $\varphi(\cdot)$ has been defined in (\ref{whatisf}). This
reflects the fact that the only change from before is in the number of
children on the root (used to be $d-1$, now is $d$).

Substituting the value of $\varphi$ from (\ref{rootroot}) we get
\begin{eqnarray*}
\biggl( \frac{1}{\sqrt{d-1}}H_d-z \biggr)^{-1}_{\Root, \Root
}&=&\biggl( -z+ \frac{d}{d-1}\frac{U_{\zeta-1}(z/2)}{U_{\zeta}(z/2)}
\biggr)^{-1}\\
&=& \frac{U_{\zeta}(z/2)}{-z U_{\zeta}(z/2) + U_{\zeta-1}(z/2) +
(d-1)^{-1} U_{\zeta-1}(z/2)}\\
&=& - \frac{U_{\zeta}(z/2)}{U_{\zeta+1}(z/2) - (d-1)^{-1} U_{\zeta-1}(z/2)}.
\end{eqnarray*}
The final step above follows from the recursion of the Chebyshev
polynomials given in (\ref{chebyrec}). This proves (i).
The proof of (ii) follows by a similar argument.
\end{pf}

Recall that we are ultimately interested in how close the Stieltjes
transform of the $d$-regular graph is to the Stieltjes transform of the
semicircle,~$s(z)$. Toward this goal, we will need some estimates for
the functions defined in Lemma~\ref{lemmatreeform2}.
\begin{lemma}\label{variousestimates}
Consider the functions defined in Lemmas~\ref{treeform} and \ref
{lemmatreeform2}. Then for all $z$ such that $z/2 \in E_r = \{ y
\in\mathbb{C}\dvtx \vert y+ \sqrt{y^2 -1} \vert=r \}$, for some $r
> 1$ such that $r^{-\zeta}<1/2$, one has the following estimates:
\begin{longlist}[(iii)]
\item[(i)] Consider $\varphi_d(\zeta)$ and $s(z) = -(z - \sqrt
{z^2-4})/2$. We have the following estimate:
%
%
\begin{equation}\label{estvarphid}
\vert\varphi_d(\zeta) - s(z) \vert \le C_0\biggl[ \frac{2r^{-2\zeta}}{1
- r^{-2\zeta-2}} + \frac{1}{d-1}\biggr],
\end{equation}
where $C_0$ is a constant.
%
\item[(ii)] The following bound on $\psi(\cdot)$ holds:
%
%
\begin{equation}\label{estpsi}
\vert\psi(\zeta) \vert \le\frac{r^{-\zeta-1}}{(d-1)^{\zeta/2}}
\frac
{2}{1 - r^{-2\zeta-4}}.
\end{equation}
\item[(iii)] Similarly,
%
%
\begin{equation}\label{estpsid}
\vert\psi_d(\zeta) \vert \le C_0 \frac{r^{-\zeta-1}}{(d-1)^{\zeta/2}}
\frac{1}{1 - r^{-2\zeta-4}}.\vadjust{\goodbreak}
\end{equation}
\end{longlist}
\end{lemma}
\begin{rmk}
The condition $r^{-\zeta} < 1/2$ is a priori more restrictive than
necessary for the purposes of Lemma~\ref{variousestimates}; since we
will in fact be interested in the case when $r^{-\zeta} = o(1)$, this
does not matter.
\end{rmk}
\begin{pf*}{Proof of Lemma~\ref{variousestimates}}
(i) Let $\omega= (z+\sqrt{z^2-4})/2= -1/s(z)$. Note that
$\vert\omega \vert=r > 1$. Recall from (\ref{drootroot}) and (\ref
{whatisun}) that
\begin{eqnarray*}
\varphi_d(\zeta)&=&\frac{\omega^{\zeta+1} - \omega^{-\zeta
-1}}{\omega^{\zeta+2} - \omega^{-\zeta-2} - ({1}/({d-1}))(
\omega^{\zeta} - \omega^{-\zeta} )}\\
&=& \frac{1}{\omega}
\frac{1 - \omega^{-\zeta-2}}{1 - \omega^{-2\zeta-4} -
({1}/({d-1}))( \omega^{-2} - \omega^{-2\zeta-2})}.
\end{eqnarray*}
Thus
\[
\vert\varphi_d(\zeta) \vert \le r^{-1}\frac{1 + r^{-\zeta-2}}{1 -
r^{-2\zeta-4} - (d-1)^{-1}(r^{-2} + r^{-2\zeta-2}) }
\]
as long as the right-hand side above is positive. Since $r > 1$ we get
$\varphi_d(\zeta)$ is bounded by an absolute constant $C_0$ when
$r^{-\zeta}< 1/2$.

We now estimate the quantity $\varphi(\zeta)$. By (\ref{rootroot})
and (\ref{whatisun}) we get
%
%
\begin{equation}\label{varphidconstant}
\varphi(\zeta)= -\frac{\omega^{\zeta+1} - \omega^{-\zeta
-1}}{\omega^{\zeta+2} - \omega^{-\zeta-2}}=-\omega^{-1} \frac{1 -
\omega^{-2\zeta-2}}{1-\omega^{-2\zeta-4}}.
\end{equation}
Thus
\begin{eqnarray*}
\vert\varphi(\zeta) + \omega^{-1} \vert&=& \vert\varphi(\zeta) -
s(z) \vert
= \biggl\vert\omega^{-1}\biggl( 1 - \frac{1-\omega^{-2\zeta-2}}{1 - \omega
^{-2\zeta-4}} \biggr) \biggr\vert\\
&=& r^{-1} \biggl\vert\frac{\omega^{-2\zeta -2} -
\omega^{-2\zeta-4}}{1 - \omega^{-2\zeta-4}} \biggr\vert\\
&=& r^{-2\zeta-3} \biggl\vert\frac{1 - \omega^{-2}}{1 - \omega^{-2\zeta
-4}} \biggr\vert\le\frac{r^{-2\zeta-3}}{1 - r^{-2\zeta-4}}.
\end{eqnarray*}
%

To get to $\varphi_d$, consider formula (\ref{dfromnond}). Note that
\[
-\frac{1}{z-s(z)}=-\frac{1}{z/2 + \sqrt{z^2 -4}/2 }=-\omega^{-1}=s(z).
\]
Thus
\[
\varphi_d(\zeta) - s(z) = - \frac{1}{z + ({d}/({d-1})) \varphi
(\zeta-1)} + \frac{1}{z - s(z)}.
\]
Hence
\begin{eqnarray*}
\vert\varphi_d(\zeta) - s(z) \vert &=& \biggl\vert s(z) - \frac{d}{d-1}
\varphi (\zeta-1) \biggr\vert\biggl\vert\frac{1}{z - s(z)} \biggr\vert
\biggl\vert\frac
{1}{z + ({d}/({d-1})) \varphi(\zeta-1)} \biggr\vert\\
&=&\biggl\vert\frac{d}{d-1}\bigl(s(z) - \varphi(\zeta-1) \bigr) - \frac{s(z)}{d-1}
\biggr\vert \vert\omega^{-1} \vert \vert\varphi_d(\zeta) \vert\\
&\le& C_0 \biggl[ \biggl(\frac{d}{d-1}\biggr)\frac{r^{-2\zeta-1}}{1
- r^{-2\zeta-2}} + \frac{r^{-1}}{d-1}\biggr] r^{-1}.
\end{eqnarray*}
Since $r >1$ this completes the proof.

\mbox{}\hphantom{i}(ii) We use the estimate on the Chebyshev polynomials given in
(\ref{chebyest}) and our assumption on $z$ to get
\begin{eqnarray*}
\vert\psi(\zeta) \vert &=& \frac{1}{(d-1)^{\zeta/2}} \biggl\vert\frac
{1}{U_{\zeta+1}(z/2)} \biggr\vert\le\frac{1}{(d-1)^{\zeta/2}} \frac
{r+r^{-1}}{r^{\zeta+2} - r^{-\zeta-2}}\\
&=&\frac{r^{-\zeta
-1}}{(d-1)^{\zeta/2}} \frac{1+r^{-2}}{1 - r^{-2\zeta-4}}.
\end{eqnarray*}

(iii)
This part is similar, since
\begin{eqnarray*}
\vert\psi_d(\zeta) \vert &=& \frac{1}{(d-1)^{\zeta/2}} \biggl\vert
\frac {1}{U_{\zeta+1}(z/2)} \biggr\vert \biggl\vert\frac{1}{1 - (d-1)^{-1}
U_{\zeta -1}(z/2)/U_{\zeta+1}(z/2)} \biggr\vert\\
&\le&\frac{r^{-\zeta-1}}{(d-1)^{\zeta/2}} \frac{1+r^{-2}}{1 -
r^{-2\zeta-4}} \biggl\{ \frac{1}{1 - (d-1)^{-1} \vert U_{\zeta
-1}(z/2)/U_{\zeta+1}(z/2) \vert}\biggr\}.
\end{eqnarray*}
Note that, by way of its definition, the constant $C_0$ in part (i) is
an upper bound on the final term, hence the estimate.
%
%
%
%
%
%
\end{pf*}

\subsection{From trees to regular graphs} \label{local2}

Consider now a (deterministic) $d$-regular graph $G$ with a
distinguished vertex called the root such that, for some $\zeta\ge1$,
the $(\zeta+1)$-neighborhood of the root is a tree. That is to say,
consider the subgraph consisting of all vertices in $G$ whose distance
from the root is at most $\zeta+1$ and the edges between them; we
assume that this subgraph has no cycles.

This gives us a natural partition of the graph. We denote the tree
subgraph induced by the root and all vertices of distance at most
$\zeta$ from the root by~$\tree_d$. We denote the ``boundary'' of
this graph, that is, the set of vertices that are at distance exactly
$\zeta$ from the root, by $\partial\tree_d$. The subgraph induced by
the vertices in the complement of $\tree_d$ will be denoted by $\tree
_d^c$, and its own boundary, that is, the set of vertices at distance
exactly $\zeta+1$ from the root, will be denote by $\partial\tree
_d^c$. For further clarification, please refer to Figure~\ref{treebdrys}.

%
\begin{figure}

\includegraphics{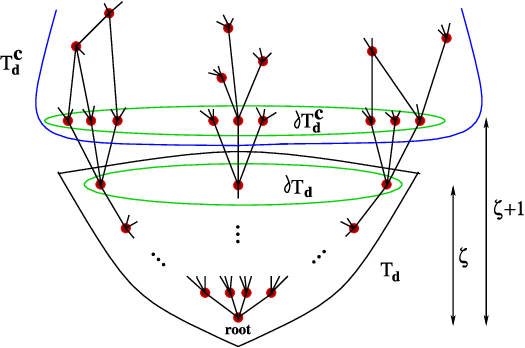}

\caption{Illustration\vspace*{1pt} of $\tree_d$, $\tree_d^c$, $\partial\tree_d$
and $\partial\tree_d^c$ starting at a vertex \textup{root}, with
maximal tree depth~$\zeta$. Note that the neighbors in $\tree_{d}^c$
of the vertices from $\partial\tree_d^c$ may, but need not continue
the tree-like structure. For this particular tree, $d=4$.} \label{treebdrys}
\end{figure}

Note that all edges between $\tree_d$ and $\tree_d^c$ are between
$\partial\tree_d$ and $\partial\tree_d^c$.

Additionally we will denote the set of vertices of $\tree_d$ (resp.,
$\tree_d^c$) by $V(\tree_d)$ [resp., $V(\tree_d^c)$].

Let $A$ denote the adjacency matrix of this graph, and let $H_d$ denote
the adjacency matrix of the subgraph $\tree_d$. Label the vertices of
$H_d$ as in Lemmas~\ref{treeform} and~\ref{lemmatreeform2}, and
write $A$ in the block matrix form
\[
A=
\left[\matrix{
D & B\cr
B' & H_d}\right].
\]
Here $D$ is the adjacency matrix of $\tree_d^c$, the matrix $B$
records only and all the edges between $\partial\tree_d$ and
$\partial\tree_d^c$ and we again use the notation $B'$ for the
transpose of
$B$.

We will now proceed to estimate how close the Stieltjes transform of
the regular graph is to that of the tree.
\begin{lemma}\label{estimateepsilon}
Fix a complex number $z$ such that $z \in E_r$ [see (\ref{whatiser})]
for some $r > 1$. Let $\varepsilon$ denote the quantity
\[
\varepsilon= \biggl( \frac{1}{\sqrt{d-1}}A - z \biggr)^{-1}_{\Root,
\Root} - \biggl( \frac{1}{\sqrt{d-1}} H_d -z \biggr)^{-1}_{\Root,
\Root}.
\]
We have the following bound:
\[
\vert\varepsilon \vert \le\biggl( \frac{2C_0^2}{ 1 - r^{-2\zeta-4}}
\biggr) \frac{r^{-2\zeta-2}}{\Im(z)}.
\]
\end{lemma}
\begin{pf} Define the vector
\[
v=B\biggl( \frac{1}{\sqrt{d-1}} H_d - z \biggr)^{-1}e_{\Root},
\]
where $e_{\Root}$ is the vector that puts mass at the root vertex and
zero elsewhere.

Then by using the formula for the inverse of block matrices (\ref
{invform2}) we get
%
%
\begin{equation}\label{epsilonlong}
\varepsilon= \frac{1}{d-1} \biggl\{ v'\biggl( \frac{1}{\sqrt{d-1}}D-z
- \frac{1}{d-1}B \biggl( \frac{1}{\sqrt{d-1}}H_d - z \biggr)^{-1}B'
\biggr)^{-1}v \biggr\}.\hspace*{-35pt}
\end{equation}

Our first job is to estimate the elements of the vector $v$. From the
definition it is clear that the rows of $B$ (and columns of $B'$) are
labeled by the vertices of $\tree_d^c$. We write $B_{i*}$ to designate
the $i$th row of $B$.

We obtain
%
%
\begin{eqnarray}\label{definev}
v_i &=& B_{i*}\biggl( \frac{1}{\sqrt{d-1}} H_d - z
\biggr)^{-1}e_{\Root} = \sum_{k \in V(\tree_d)} B_{ik} \biggl( \frac
{1}{\sqrt{d-1}} H_d - z \biggr)^{-1}_{k, \Root}\nonumber\hspace*{-35pt}\\[-8pt]\\[-8pt]
&=&\sum_{k \in V(\tree_d^c)} B_{ik} \biggl( \frac{1}{\sqrt{d-1}} H_d
- z \biggr)^{-1}_{\Root,k}\qquad \mbox{by symmetry of $H_d$}.
\nonumber\hspace*{-35pt}
\end{eqnarray}

Note that $B_{ik}$ is positive (i.e., $1$) if and only if $i \in
\partial\tree_d^c$, $k \in\partial\tree_d$ and $i$ and~$k$ have an
edge between them. Thus:
\begin{longlist}[(iii)]
\item[(i)] $v_i=0$ unless $i \in\partial\tree^c_d$.
\item[(ii)] When $i \in\partial\tree_d^c$,
%
%
\begin{equation}\label{whatisvi}
v_i = \sum_{k \in\partial\tree_d, k\sim i} \biggl( \frac{1}{\sqrt
{d-1}} H_d - z \biggr)^{-1}_{\Root,k}= \psi_d(\zeta).
\end{equation}
Here $\psi_d(\zeta)$ has been defined in (\ref{drootleaf}). The fact
that there is exactly one $k \in\partial\tree_d$ such that $i \sim
k$ follows from our assumption that the $(\zeta+1)$ neighborhood of
the root is a tree (see Figure~\ref{treebdrys}).
\item[(iii)] By counting the number of elements in $\partial\tree
_d^c$, we get
%
%
\begin{equation}\label{normv}
\Vert v \Vert^2= d(d-1)^\zeta\vert\psi_d(\zeta) \vert^2\le C_0^2
\frac{d
r^{-2\zeta-2}}{1 - r^{-2\zeta-4}},
\end{equation}
where the final estimate is from (\ref{estpsid}).
\end{longlist}

Note that the matrix
\[
\biggl( \frac{1}{\sqrt{d-1}}D-z - \frac{1}{d-1}B \biggl( \frac
{1}{\sqrt{d-1}}H_d - z \biggr)^{-1}B' \biggr)
\]
is precisely the matrix $\mathbb{G}$ appearing in (\ref{invform2}).
In particular, it is the top left block of the matrix $
[(d-1)^{-1/2}A - z]^{-1}$. Hence, by padding the vector $v$ with
extra zeros, we get a vector $\bar v$ such that
\[
\varepsilon= \frac{1}{d-1} \bar v' \biggl( \frac{1}{\sqrt{d-1}}A -z
\biggr)^{-1}\bar v.
\]

Since the matrix $A$ is real symmetric, it has only real eigenvalues.
It follows from spectral decomposition of the real, symmetric matrix
$A$, that\vadjust{\goodbreak} for any real vector $y$
%
%
\begin{equation}
y'\biggl( \frac{1}{\sqrt{d-1}}A - z \biggr)^{-1}y \le\frac{\Vert y \Vert
^2}{\Im(z)}.
\end{equation}

Observe from (\ref{whatisvi}) that $v=\psi_d(\zeta)e$, where $e$ is
a real vector of ones and zeroes which is one precisely for the labels
corresponding to $\partial\tree_d$.

Combining our previous observations, we get
\[
v'\biggl( z - \frac{1}{\sqrt{d-1}} D - \frac{1}{d-1} \chi
\biggr)^{-1}v = \psi_d^2(\zeta) e'\biggl( \frac{1}{\sqrt{d-1}} A -
z\biggr)^{-1}e.
\]

Now using the bound from (\ref{normv}) in (\ref{epsilonlong}) we get
\begin{eqnarray*}
\vert\varepsilon \vert &\le& \frac{\vert\psi_d^2(\zeta) \vert
}{d-1}\frac{\Vert e \Vert^2}{\Im(z)}=\frac{1}{d-1}\frac{\Vert v
\Vert^2}{\Im(z)} \le
C_0^2\frac{d}{d-1} \biggl( \frac{r^{-2\zeta-2}}{1 - r^{-2\zeta
-4}}\biggr)\frac{1}{\Im(z)}\\
&\le& \biggl( \frac{2C_0^2}{ 1 - r^{-2\zeta-4}} \biggr) \frac
{r^{-2\zeta-2}}{\Im(z)}.
\end{eqnarray*}
This completes the proof of the lemma.
\end{pf}

We now arrive at our main result about \textit{deterministic} regular
graphs. Consider the set-up as in Lemma~\ref{estimateepsilon}. Now a
consider a sequence of graphs~$G_n$ such that $G_n$ is $d_n$-regular.
Each $G_n$ has a marked vertex called the root such that for some
sequence $\{\zeta_n\}$, the $\zeta_n+1$ neighborhood of the root is
acyclic in $G_n$.
\begin{lemma}\label{thmtriangle}
Assume that that sequences $\{d_n\}$ and $\{\zeta_n\}$ both tend to
infinity with $n$ with the following restriction. There exists a
sequence $r_n$ such that
%
%
\begin{equation}\label{mainassumption}
r_n= e^{d_n^{-\alpha}}\qquad \mbox{for some $0 < \alpha< 1$}\quad
\mbox{and}\quad
r_n^{-\zeta_n} = o(1/d_n).
\end{equation}
%
Let $U_n$ denote all complex numbers $z$ such that $\Im(z) > (r_n-r^1_n)/2$.
If $A_n$ denotes the adjacency matrix of the graph $G_n$, for all $d_n$ such
that $d_n \geq (\log 2)^{1/\alpha}$, then for all $z$, eventually in
every $U_n$, we have
%
%
\begin{equation}\label{mainbound}
\biggl| \biggl( \frac{1}{\sqrt{d_n-1}} A_n - z \biggr)^{-1}_{\Root,
\Root} - s(z) \biggr| = O( 1/d_n ),
\end{equation}
where $s(z)$ is the Stieltjes transform of the semicircle law
$s(z)=-z/2 + \sqrt{z^2-4}/2$.
\end{lemma}
\begin{pf} Under our assumptions certain simplifications are immediate.
The constant $C_0$ appearing in Lemma~\ref{variousestimates} (and
later) can be taken to be an absolute constant. Since $r_n^{-\zeta_n}=
o(1/d_n)$, we can choose $C_0$ large enough in~(\ref{estvarphid}) such that
%
%
\begin{equation}\label{diffpart1}
\vert\varphi_{d_n}(\zeta_n)-s(z_n) \vert\le C_0/d_n.\vadjust{\goodbreak}
\end{equation}
Here $z_n$ is any sequence of complex numbers such that $z_n/2$ belongs
to the ellipse $E_{r_n}$ defined in (\ref{whatiser}).

Note that
\[
\varphi_{d_n}(\zeta_n)= \biggl( \frac{1}{\sqrt{d_n-1}} A_n - z_n
\biggr)^{-1}_{\Root, \Root}.
\]
We now use Lemma~\ref{estimateepsilon}.

Consider now
\[
\Im(z) \geq \frac{r_n - r_n^{-1}}{2} = \frac{1}{2}[
e^{d_n^{-\alpha}} - e^{- d_n^{-\alpha}} ].
\]
One can easily verify the inequality
%
%
\begin{equation}\label{ineqsinh}
\frac{e^{x} - e^{-x}}{2} \ge x/2\qquad \mbox{for all } x\le\log2.
\end{equation}
Applying this inequality above, we get $\Im(z)\ge d_n^{-\alpha}/2$
for all $d_n$ as stated in the lemma.

Combining with Lemma~\ref{estimateepsilon} we get
\[
\vert\varepsilon \vert \le2 C_0^2 \frac{r_n^{-2\zeta_n}}{1 -
r_n^{-2\zeta
_n - 4}} d_n^\alpha\le C_1 \frac{1}{d_n^2} d_n^\alpha= o(
1/d_n ).
\]
Combining with (\ref{diffpart1}) this completes the proof of the theorem.
\end{pf}

We will need the following lemma to show that the previous ``tree
approximation'' result about \textit{deterministic} regular graphs can
be applied to the random graph by choosing \textit{almost any} vertex as
the ``root.''
\begin{lemma}\label{treeapprox}
Let $d_n= (\log n)^\gamma$, for some positive $\gamma$. Let $\{r_n\}$
be as in Lem\-ma~\ref{thmtriangle}.
Let $\eta_n$ be defined as $(r_n - r_n^{-1})/2$. For some $\beta> 1$,
we define the sequence $\{\zeta_n, n=1,2,\ldots\}$ satisfying
%
%
\begin{equation}\label{condrn}
\zeta_n= \frac{1}{4}\frac{\log n}{\log(d_n-1)} - \beta.
\end{equation}

Let $J(n)$ be the set of vertices in $G_n$ whose $\zeta
_n$-neighborhoods are acyclic and let $\vert J(n) \vert$ denote its size.
Let $\Omega(n)$ be the event
%
%
\begin{equation}\label{whatisomega2}
{\vert J(n) \vert}/{n} > 1- \eta_n/d_n.
\end{equation}
One can choose $\beta$ no larger than $4$ such that
%
%
\begin{equation}\label{evecnt2}
P( \Omega^c(n))\le o\biggl( \frac{1}{n} \biggr).
\end{equation}
\end{lemma}
\begin{pf} The proof of this result is very similar to the proof of
Lemma~\ref{treeapp1}. We again use the estimates of McKay, Wormald and
Wysocka~\cite{mww04} on the Poisson approximation to the number of
short cycles. Consider a sequence~$\{\zeta_n\}$ as in the statement.
It is clear from the choice that
%
%
\begin{equation}\label{someeq}
4\zeta_n\log(d_n-1) = \log n - \beta\log(d_n-1)\vadjust{\goodbreak}
\end{equation}
or
\begin{equation}
(d_n-1)^{4\zeta_n} = \frac{n}{(d_n-1)^{\beta}}=o(n).
\end{equation}
Thus, we can take $g=2\zeta_n$ in (\ref{whatisgn}).

The argument is essentially the same as the one used in the proof of
Lemma~\ref{treeapp1}; rather than repeating it, we choose to only
highlight the differences.


The total number of vertices whose $\zeta_n$-neighborhoods are not
acyclic can be bounded above by
\[
N^*_{\zeta_n}=\sum_{s=3}^{2\zeta_n} s (d_n-1)^{(2\zeta_n-s)/2} M_s,
\]
where $M_s$ is the number of cycles of length $s$.

Taking expectations and variances above we get
\[
E N^*_{\zeta_n} = O\bigl( (d_n-1)^{2\zeta_n} \bigr)
\]
and
\[
\var(N^*_{\zeta_n}) \le \zeta_n^2(2\zeta_n+1) (d_n-1)^{2\zeta
_n} + \zeta_n^2 (2\zeta_n+d_n) o(1).
\]

We now want to use Chebyshev's inequality on the quantity $1-\vert J(n)
\vert/n$.
\begin{eqnarray*}
P\biggl( 1 - \frac{\vert J(n) \vert}{n} > \frac{\eta_n}{d_n} \biggr) &
\le& P( N^*_{\zeta_n} > {n\eta_n}/{d_n} )
\le \frac{d_n^2}{n^2\eta_n^2} E( N^*_{\zeta_n} )^2
\\
& \le& \frac{3d_n^2\zeta_n^3 (d_n-1)^{2\zeta_n}}{n^2\eta_n^2} +
\frac{3\zeta_n^3d_n^2o(1)}{n^2\eta_n^2} \\
&&{}+ \frac{d_n^2}{n\eta
_n^2}O\biggl( \frac{ (d_n-1)^{4\zeta_n}}{n} \biggr).
\end{eqnarray*}

Let us analyze the three terms that appear above. We use inequality
(\ref{ineqsinh}) to obtain $\eta_n \ge d_n^{-\alpha}/2$. Using
(\ref{someeq}) we get
\begin{eqnarray*}
\frac{3d_n^2\zeta_n^3 (d_n-1)^{2\zeta_n}}{n^2\eta_n^2} &\le& 3 d_n^2
\zeta_n^3 \frac{\sqrt{n}}{(d_n-1)^{\beta/2}} \frac{1}{n^2} 4
d_n^{2\alpha}= \frac{\mathrm{poly}\mbox{-}\mathrm{log}\
n}{n^{3/2}},\\
\frac{3\zeta_n^3 d_n^2 o(1)}{n^2 \eta_n^2} &\le& \frac{12\zeta_n^3
d_n^{2+ 2\alpha} o(1)}{n^2}= \frac{\mathrm{poly}\mbox{-}\mathrm
{log}\ n}{n^{2}} o(1).
\end{eqnarray*}
The leading term is the last term on the right which is of the order
$d_n^{2-\beta}/(n\eta_n^2)$, because by (\ref{someeq}) we get
$(d_n-1)^{4\zeta_n}/n= (d_n-1)^{-\beta}$.

We now choose $\beta> 2\alpha+2$ (and thus no larger than $4$, since
$\alpha<1$) such that
\[
\frac{d_n^{2-\beta}}{\eta_n^2} \le4d_n^{2-\beta} d_n^{2\alpha} =o(1).
\]
This completes the proof of the lemma.
\end{pf}

\subsection{The Stieltjes transform} \label{local3}

We bring now all the results of the previous sections together. Below
is the first theorem of this section.
\begin{theorem}\label{stielapprox}
Consider a sequence of $d_n$-regular graphs on $n$ vertices where $d_n=
(\log n)^\gamma$ for\vadjust{\goodbreak} some $\gamma> 0$. Let $A_n$ be the adjacency
matrix of the graph, and consider the Stieltjes transform
\[
s_n(z) = \tr\biggl( \frac{1}{\sqrt{d_n-1}} A_n - z \biggr)^{-1},\qquad
\Im(z) > 0.
\]
Let $\eta_n=(r_n - r_n^{-1})/2$ where $r_n = \exp( d_n^{-\alpha
} )$ for some $0 < \alpha< \min(1, 1/\gamma)$.

Let $U_{n}$ denote all complex numbers $z$ such that $\Im(z) >\eta
_n$. Then there is a large enough constant $C > 0$ such that the
Stieltjes transform of the empirical eigenvalue distribution of the
$n\times n$ matrix $A_n$ satisfies
\[
P\Bigl( \sup_{z\in U_{n}}\vert s_n(z) - s(z) \vert > {C}/{d_n} \Bigr)
\le o( 1/n ).
\]
Here $s(z)$ refers to the Stieltjes transform of the semicircle law.
\end{theorem}
\begin{pf}
We condition on the event $\Omega(n)$ which, from Lemma \ref
{treeapprox}, happens with a probability of at least $1-o(1/n)$.

Consider the set $J(n)$ from the Lemma~\ref{treeapprox} and write
%
%
\begin{equation}\label{addjn}
s_n(z) = \frac{1}{n} \sum_{k\in J(n)} \biggl( \frac{1}{\sqrt
{d_n-1}} A_n - z \biggr)^{-1}_{k,k} + \frac{1}{n} \sum_{k\notin
J(n)} \biggl( \frac{1}{\sqrt{d_n-1}} A_n - z \biggr)^{-1}_{k,k}.\hspace*{-35pt}
\end{equation}

Now, the spectral norm of $(\frac{1}{\sqrt{d_n-1}} A_n - z
)^{-1}$ is bounded above by $1/\Im(z)$, which in turn is
bounded above by $\eta_n^{-1}$ for all $z \in U_n$. Thus, for all $z
\in U_n$ and all $k$, we have the obvious bound $(\frac{1}{\sqrt
{d_n-1}} A_n - z )^{-1}_{k,k}\le\eta_n^{-1}$.

Summing up over all $k \notin J(n)$, we get
\[
\frac{1}{n} \sum_{k\notin J(n)} \biggl\vert\biggl( \frac{1}{\sqrt {d_n-1}}A_n
- z \biggr)^{-1}_{k,k} \biggr\vert \le\frac{n-\vert J(n) \vert}{n\eta
_n} \le\frac{1}{d_n},
\]
where the final inequality holds on the event $\Omega$. Thus, we get
\[
s_n(z) = \frac{1}{n} \sum_{k\in J(n)} \biggl( \frac{1}{\sqrt
{d_n-1}} A_n - z \biggr)^{-1}_{k,k} + O\biggl( \frac{1}{d_n} \biggr).
\]

Now, on the event $\Omega(n)$, every $k \in J(n)$ can be considered as
the root for a $d_n$-regular tree of depth $\zeta_n$. We can now apply
Lemmas~\ref{estimateepsilon} and~\ref{thmtriangle}. Note that,
technically, the $\zeta_n$ in Lemma~\ref{treeapprox} and the $\zeta$
used in Lemma~\ref{estimateepsilon} differ by at most one; however,
this does not affect the following calculations.

We verify the following assumptions from Lemma~\ref{thmtriangle}:
\begin{eqnarray*}
r_n^{-\zeta_n} &=& \exp\biggl( -\frac{1}{4}d_n^{-\alpha}\frac{\log
n}{\log(d_n-1)} + \beta d_n^{-\alpha} \biggr)\\
&=& \exp\biggl( -\frac{1}{4}\frac{(\log n)^{1-\alpha\gamma}}{\log
(d_n-1)} + \frac{\beta}{(\log n)^{\alpha}} \biggr) \\
& \le& C \exp\biggl(- \frac{(\log n)^{1- \alpha\gamma}}{4\gamma
\log\log n} \biggr) \ll\exp( - \gamma\log\log n ) =
\frac{1}{d_n},
\end{eqnarray*}
whenever $\alpha\gamma<1$.

We now combine our error estimate (\ref{mainbound}) to see that for
any $k\in J(n)$, on the event $\Omega(n)$, and for all $z\in U_n$, we have
%
%
\begin{equation}\label{errork1}
\biggl\vert\biggl(\frac{1}{\sqrt{d_n-1}} A_n-z\biggr)^{-1}(k,k) - s(z)
\biggr\vert = O\biggl( \frac{1}{d_n} \biggr),
\end{equation}
where the constants in the $O(\cdot)$ above does not depend on $k$.

Combining this with decomposition (\ref{addjn}) we get
%
%
\begin{equation}\label{diff7}
\sup_{z\in U_n}\vert s_n(z) - s(z) \vert = O\biggl( \frac{1}{d_n} \biggr)
\end{equation}
on the event $\Omega(n)$, which holds with probability at least $1 -
o(1/n)$. This completes the proof of the lemma.
\end{pf}

Recall now the setup of Theorem~\ref{localsemic}. Fix $\delta>0$.
Let $d_n, \eta_n$ be as in Theorem~\ref{stielapprox}. Then we will
show that there exists an $N$ large enough such that for all $n \geq
N$, for any interval $I \subset\mathbb{R}$ of length $|I| \geq\max\{
2\eta_n, \eta_n / (-\delta\log\delta)\}$,
\[
\biggl\vert\ecount_I - n \int_I \dsc(x) \,dx \biggr\vert < \delta n \vert I
\vert
\]
with probability at least $1-o(1/n)$.

Here $\ecount_I$ is the number of eigenvalues of $W_n$ in the interval
$I$, and $\dsc$ refers to the density of the semicircle law as in
(\ref{whatisfsc}).
\begin{pf*}{Proof of Theorem~\ref{localsemic}}
Theorem~\ref{stielapprox} leads to Theorem~\ref{localsemic} whose
proof follows almost identically to Lemma 60 in the article by Tao and
Vu~\cite{taovu}. The only major difference between our theorem and
Lemma 60 of~\cite{taovu} is the fact that our interval~$I$ can lie
anywhere on the real line and is not restricted to a subset of $(-2,2)$
as Lemma 60 requires. We provide an outline of the argument but skip
the details.

The idea lies in the observation that a good control over the Stieltjes
transform near the real line allows one to invert the transform and
have an estimate of the empirical spectral density. This is due to the
following inversion formula:\vadjust{\goodbreak} if $G$ is a continuous distribution on the
real line with Stieltjes transform $s_G$, one gets
\[
G[a,b] = \lim_{\varepsilon\rightarrow0} \frac{1}{\pi} \int_a^b \Im
\bigl( s_G(x+i\varepsilon) \bigr)\,dx.
\]

Fix an interval $I \subseteq[-2,2]$ such that $\vert I \vert\ge2\eta$.
Define the function
\[
F(y) = \frac{1}{\pi} \int_I \frac{\eta_n}{\eta_n^2 + (y-x)^2}\,dx.
\]
Then it follows that if $\lambda_i$ denotes the $i$th eigenvalues of
the matrix $\frac{1}{\sqrt{d_n-1}}A_n$, then
\begin{eqnarray*}
\frac{1}{n}\sum_{i=1}^n F(\lambda_i) &=& \frac{1}{\pi}\int_I \frac
{1}{n}\sum_{i=1}^n \frac{\eta_n}{\eta_n^2 + (\lambda_i-x)^2}\,dx\\
&=&\frac{1}{\pi}\int_I \Im\bigl( s_n(x+i\eta_n) \bigr)\,dx.
\end{eqnarray*}
Also, if $\dsc$ denotes the density of the semicircle law, we get
\[
\int_{-2}^2 F(y) \dsc(y)\,dy = \frac{1}{\pi} \int_I \Im\bigl(
s(x+i \eta_n) \bigr)\,dx.
\]

Using the approximation between $s_n$ and $s$ obtained in Theorem \ref
{stielapprox}, we get
\begin{eqnarray*}
\Biggl\vert\frac{1}{n}\sum_{i=1}^n F(\lambda_i) - \int_{-2}^2 F(y)
\dsc(y)\,dy \Biggr\vert &\le& \frac{1}{\pi}\int_I \vert s_n(x+i\eta_n)-
s(x+i\eta_n) \vert \,dx \\
&\le& C\frac{\vert I \vert}{d_n}\qquad \mbox{with probability } 1 -
o\biggl( \frac{1}{n}\biggr).
\end{eqnarray*}
Choose $n$ large enough such that $1/d_n \le\delta$ for all
subsequent $n$. Thus
\[
\frac{1}{n}\sum_{i=1}^n F(\lambda_i) = \int_{-2}^2 F(y) \dsc(y)\,dy
+ O( \delta\vert I \vert )
\]
with probability $1 - o( {1}/{n})$.

Now following the bounds in~\cite{taovu}, page 60, proof of Lemma 64,
we obtain the bounds
\begin{eqnarray*}
\int_{-2}^2 F(y) \dsc(y) \,dy &=& \int_I \dsc \,dy + O\biggl( \eta_n
\log\frac{\vert I \vert}{\eta_n} \biggr),\\
\frac{1}{n}\sum_{i=1}^n F( \lambda_i ) &=& \frac
{\ecount}{n} + O\biggl( \eta_n \log\frac{\vert I \vert}{\eta_n} \biggr).
\end{eqnarray*}

Putting all these together we obtain that with probability $1 -
o(1/n)$, one has
\[
{\ecount_I} - n\int_I \dsc \,dy = O(n\delta\vert I \vert) +
O\biggl( n\eta_n \log\frac{\vert I \vert}{\eta_n} \biggr).
\]
Finally, as observed in~\cite{taovu}, the latter term can be absorbed
in the former since $\vert I \vert\ge\eta_n/\delta\log(1/\delta
)$. This
completes the proof of the theorem.
\end{pf*}

\section{Delocalization of eigenvectors}\label{eigenvec}

Closely related to approximation of the empirical spectral distribution
is the fact that the $\ltwo$-norm of a normalized eigenvector
restricted to a \textit{large} subset of the vertices cannot be
\textit{small}. We follow Definition~\ref{deloc} and prove Theorem
\ref{deloc}.

Recall the set-up of Theorem~\ref{localsemic}. Fix $\delta>0$.

Let $T_n \subseteq\{1,2,\ldots,n\}$ be a sequence of sets of size
$L_n = o(\eta_n^{-1})$. Let $\Omega_1(n)$ be the event that some
unit-norm eigenvector of the matrix $A_n$ is $(T_n,\delta)$ localized;
that is,
\[
\Omega_1(n) = \{ \exists i\dvtx\|v_i|_{T_n}\|^2 \geq1- \delta\mbox{,
for some } v_i \mbox{ such that } A_n v_i= \lambda_i v_i \}.
\]
Then, for all sufficiently large $n$, we will show that
\[
P( (\Omega_1(n))^c ) \ge e^{-L_n \eta_n / d_n} \biggl( 1
- o\biggl( \frac{1}{d_n} \biggr)\biggr) =
1 - o\biggl( \frac{1}{d_n} \biggr).
\]
\begin{pf*}{Proof of Theorem~\ref{thmdeloc}}
Consider the set $J(n)$, defined in Lemma~\ref{treeapprox}, of
vertices whose $\zeta_n$-neighborhoods are acyclic. Recall the event
$\Omega(n)$, whose probability (as soon as $n$ is large enough) is $1
- o(1/n)$, which is that $|J(n)|/n > 1 - \eta_n/d_n$.

We first prove part (i). The event $\Omega_1(n)$ can be decomposed as
two disjoint events depending on whether the set $T_n$ is a subset of
$J(n)$ or not. We first examine the event
\[
\Omega'_1(n)= \{ T_n \subseteq J(n) \} \cap\Omega_1(n)
\]
for purposes of exclusion.

Assume $\omega\in\Omega(n) \cap\Omega'_1(n)$, and fix $i, v_i$
depending on $\omega$ and $n$. The matrix $A_n$ has $n$ real
eigenvalues and corresponding\vspace*{2pt} eigenspaces. The top eigenvalue is $d_n$
with corresponding eigenvector $v_1=n^{-1/2}\one$, which is completely
delocalized. Let us now choose $v_2, v_3, \ldots, v_n$ to be a set of
normalized eigenvectors, corresponding, respectively, to the eigenvalues
$\lambda_2, \lambda_3, \ldots, \lambda_n$ of $\frac{1}{\sqrt
{d_n-1}}A_n$. Fix a subset $T_n\subseteq\{1,2,\ldots,n\}$ of size
$L_n$ as stated in the theorem.

Consider again the Stieltjes transform of the matrix $(d_n-1)^{-1/2}
A_n$ as in the proof of Theorem~\ref{stielapprox}. That is, for $z \in
\mathbb{C}$, with $\Im(z) > 0$, consider the matrix\vadjust{\goodbreak} $(\frac
{1}{\sqrt{d_n-1}}A_n-z )^{-1}$. Then by the spectral
representation it follows that for any $1\le k \le n$ we get
\[
\biggl( \frac{1}{\sqrt{d_n-1}}A_n - z \biggr)^{-1}(k,k)= \sum
_{j=1}^n \frac{v_j^2(k)}{( \lambda_j - z )}.
\]
Summing over the vertices $k \in T_n$, we get
%
%
\begin{equation}\label{gfnbnd}
\sum_{k \in T_n} \biggl( \frac{1}{\sqrt{d_n-1}}A_n - z
\biggr)^{-1}(k,k)= \sum_{j=1}^n \frac{\sum_{k \in T_n} v_j^2(k)}{(
\lambda_j - z )}.
\end{equation}

Taking the imaginary part on both sides, we get
%
%
\begin{eqnarray}\label{ineq42}
\Im\biggl( \sum_{k \in T_n} \biggl( \frac{1}{\sqrt{d_n-1}}A_n - z
\biggr)^{-1}(k,k) \biggr) &=& \sum_{j=1}^n \sum_{k \in T_n}
v_j^2(k)\Im\biggl( \frac{1}{ \lambda_j - z }\biggr)\nonumber\\
&\ge& \sum_{k \in T_n} v_i^2(k)\Im( \lambda_i - z
)^{-1} \\
&\ge&(1-\delta)\Im( \lambda_i - z )^{-1}.
\nonumber
\end{eqnarray}

Now we use the fact that for $\omega\in\Omega(n)$ and for all $k\in
J(n)$, we get from equation (\ref{errork1}) that
\[
\sup_{z\in U_n} \biggl\vert\biggl( \frac{1}{\sqrt{d_n-1}}A_n - z \biggr)^{-1}(k,k) -
s(z) \biggr\vert \le\frac{C}{d_n}
\]
for some absolute constant $C > 0$. Recall that $s(z)$, the Stieltjes
transform of the semicircle density, is given by
\[
s(z)=\tfrac{1}{2}\bigl( -z + \sqrt{z^2-4} \bigr).
\]
Thus, summing up over all $k$ in $T_n$ we get
\[
\sup_{z\in S_n} \biggl\vert\sum_{k \in T_n} \biggl( \frac{1}{\sqrt {d_n-1}}A_n
- z \biggr)^{-1}(k,k) - L_ns(z) \biggr\vert \le\frac{CL_n}{d_n}.
\]
Combining this estimate with (\ref{ineq42}) for $z$ such that $\Re
(z)=\lambda_i$ and $\Im(z)=\eta_n$ we get
\[
L_n\Im(s(z)) + CL_n/d_n \ge\eta_n^{-1}(1-\delta)
\]
for all $\omega\in\Omega(n) \cap\Omega'_1(n)$.

Since $\Im(s(z))$ is bounded and $L_n = o(\eta_n^{-1})$,
there is a large enough $N$ such that for all $n\ge N$, this inequality
will not hold when $\omega\in\Omega'_1(n)$. So we get that
%
%
\begin{equation} \label{zeroprob}
P\bigl(\Omega(n) \cap\Omega_1'(n)\bigr) = 0
\end{equation}
as soon as $n$ is large enough.

We can now write
\[
\Omega_1(n) = \bigl((\Omega(n))^c \cap\Omega_1(n) \bigr)\cup
\bigl(\Omega(n) \cap\Omega_1(n) \bigr);
\]
the probability of the first of the two events above is bounded by the
probability of $(\Omega(n))^c$ which is $o(1/n)$. Further, by (\ref
{zeroprob}), we get that
\[
\Omega(n) \cap\Omega_1(n) = \Omega(n) \cap\bigl(\Omega_1(n)
\setminus\Omega_1'(n)\bigr).
\]
Note that the last event is equivalent to saying that $\Omega(n)$ and
$\Omega_1(n)$ happen, and that in addition $T_n \not\subseteq J(n)$.

We now bound
\begin{eqnarray*}
P\bigl[ \Omega(n) \cap\bigl(\Omega_1(n) \setminus\Omega_1'(n)\bigr)
\bigr] & \leq& P[\Omega(n) \cap\{T_n \not\subseteq J(n)\} ]
\\
& = & P [ \{T_n \not\subseteq J(n)\}\mid\Omega(n) ] P
[ \Omega(n) ] \\
& = & \biggl( 1 - o \biggl( \frac{1}{n} \biggr) \biggr) P [ \{
T_n \not\subseteq J(n)\} \mid\Omega(n) ].
\end{eqnarray*}
Note now that, given the size $j = |J(n)|$, any set of $j$ labels
chosen from $\{1, \ldots, n\}$ is just as likely as any other, and
independent of the set $T_n$. So
%
%
\begin{eqnarray} \label{firstest}
P [ \{T_n \not\subseteq J(n) \} \mid|J(n)| = j ] & = & 1
- P [ \{T_n \subseteq J(n) \} \mid|J(n)| = j ] \nonumber\\
& = & 1 - \frac{{n - L_n \choose j - L_n}}{{n \choose j}} \\
& \leq& 1 - \frac{(j - L_n)^{L_n}}{n^{L_n}}.\nonumber
\end{eqnarray}
Since
\[
P [ \{T_n \not\subseteq J(n)\} \mid\Omega(n) ] = P
\biggl[ \{T_n \not\subseteq J(n)\} \Bigm|\frac{|J(n)|}{n} > 1 - \frac{\eta
_n}{d_n} \biggr],
\]
it follows from (\ref{firstest}) that
\[
P [ \{T_n \not\subseteq J(n)\} \mid\Omega(n) ] \leq
1 - \biggl(1 - \frac{\eta_n}{d_n} - \frac{L_n}{n} \biggr)^{L_n},
\]
and since $L_n = o(1/\eta_n)$ and thus $L_n/n = o(\eta_n/d_n)$, it
follows that
%
%
\begin{equation} \label{secondest}
P [ \{T_n \not\subseteq J(n)\} \mid\Omega(n) ]
\leq 1 - e^{-L_n \eta_n/d_n} \biggl( 1 - o \biggl(\frac{1}{d_n}
\biggr) \biggr). 
\end{equation}
Putting all of these together, we conclude that
\[
P [ \Omega_1(n) ] \leq1 - e^{-L_n \eta_n / d_n} \biggl(
1 - o \biggl( \frac{1}{d_n} \biggr) \biggr),
\]
%
which completes the proof.

Note that part (ii) is proved in (\ref{zeroprob}).
\end{pf*}

\begin{appendix}\label{app}
\section*{Appendix: Eigenvalues and eigenvectors of regular~trees}

Our main step in the above proofs was to understand the Stieltjes
transforms or the resolvent matrix of a finite tree where every
nonleaf vertex has $(d-1)$ children. It is quite straightforward to
compute all the eigenvalues of such a tree. Explicit eigenvalues of the
tree give us ideas about spacing distribution of eigenvalues of the
random regular graph. Fix a positive integer $d \ge2$.
\begin{lemma}\label{treeeigs}
Let $\tree$ be a finite ordered rooted tree of depth $\zeta\in
\mathbb{N}$ such that every vertex has exactly $(d-1)$ \textit
{children}. That is, the root has degree $(d-1)$, and every other
vertex, other than the leaves, has degree $d$. Let $H$ denote the
adjacency matrix of this graph.
\begin{longlist}[(ii)]
\item[(i)] Then, for any complex number $z$ the characteristic
polynomial of $H$ is given by
\[
\Delta(z;\zeta):=\det\biggl( z I - \frac{1}{\sqrt{d-1}} H \biggr)
= U_{\zeta+1}(z/2)\prod_{i=1}^{\zeta} U_{\zeta+1-i}^{(d-1)^{i} -
(d-1)^{i-1}}(z/2).
\]
\item[(ii)] The eigenvalues of the adjacency matrix are given by the
following collection. Consider $i=1,2\ldots, \zeta$: then twice the
zeros of the Chebyshev polynomial $U_{\zeta+1-i}$ appears with
multiplicity $(d-1)^i - (d-1)^{i-1}$. For $i=0$, the multiplicity is one.
\end{longlist}
\end{lemma}
\begin{pf}
Recall the recursive labeling of vertices as given in Section~\ref{local1}.

To prove conclusion (i), note that when $\zeta=0$ (i.e., the tree has
only the root vertex), the equality is trivially true. We proceed by
induction. Suppose the equality is true until depth $\zeta-1$.
Consider a tree of depth $\zeta$, and label the adjacency $H$ matrix
as above. Thus
%
%
\begin{eqnarray}
&&z-\frac{1}{\sqrt{d-1}}H \nonumber\\[-4pt]\\[-12pt]
&&\qquad=
{\fontsize{10.7pt}{11pt}\selectfont{
\left[\matrix{
z-\dfrac{1}{\sqrt{d-1}}H_1 & & & & \vspace*{2pt}\cr
& z-\dfrac{1}{\sqrt{d-1}}H_2 & & & -u \vspace*{2pt}\cr
& & \cdots& &\vspace*{2pt}\cr
& & & z-\dfrac{1}{\sqrt{d-1}}H_{d-1} &\vspace*{2pt}\cr
&& -u' && z}\right].}}\hspace*{-20pt}\nonumber
\end{eqnarray}
Here $u$ is the column vector representing the children of the root.
Notice that $u$ is $(d-1)^{-1/2}$ exactly at the $(d-1)$ coordinates
which are the last elements in each of the block matrices $H_1, \ldots
, H_{d-1}$ and zero elsewhere. The vector $u'$ is the transpose of $u$.

We now use the following well-known formula of determinant of block
matrices [akin to (\ref{invform1})]:
\[
\det
\left[\matrix{
A & B \cr
C & D}\right]
= \det(A) \det( D - C A^{-1} B ).
\]

We apply this to the matrix $z - (d-1)^{-1/2} H$ treating the the final
element~$[z]$ as one block:
note that, by our labeling, in this case $A$ is a diagonal block
matrix, and hence its determinant is a product of the determinants of
the individual blocks which are all the same and equal to $\Delta
(z;\zeta-1)$.
Thus we get
\[
\Delta(z;\zeta) = \bigl( \Delta(z;\zeta-1) \bigr)^{d-1}(z
- u' A^{-1} u ).
\]

As shown in Section~\ref{local1}, the quantity
\begin{eqnarray*}
(z - u' A^{-1} u )&=& z + \varphi(\zeta-1)= z - \frac
{U_{\zeta-1}(z/2)}{U_{\zeta}(z/2)}\\
&=& \frac{z U_{\zeta} - U_{\zeta
-1}}{U_{\zeta}}=\frac{U_{\zeta+1}(z/2)}{U_{\zeta}(z/2)}.
\end{eqnarray*}
Here $U_n$ is the Chebyshev polynomial of the second kind. Note the
reversal of sign from Section~\ref{local1} which is due to current
reversal of sign from the resolvent matrix.

Hence
\[
\Delta(z;\zeta) = \frac{U_{\zeta+1}(z/2)}{U_{\zeta}(z/2)} \bigl(
\Delta(z;\zeta-1) \bigr)^{d-1} = \prod_{i=0}^{\zeta} \biggl[\frac
{U_{\zeta+1-i}(z/2)}{U_{\zeta-i}(z/2)}\biggr]^{(d-1)^i}.
\]
The last term can be verified from the initial conditions of the
Chebyshev polynomials.

Simplifying a bit more, we get
\[
\Delta(z;\zeta) = U_{\zeta+1}(z/2)\prod_{i=1}^{\zeta} U_{\zeta
+1-i}^{(d-1)^{i} - (d-1)^{i-1}}(z/2).
\]

For part (ii), note from above that the $i$ many zeroes of $U_i$ appear
with multiplicity $(d-1)^i - (d-1)^{i-1}$. Hence the total number of
eigenvalues are
\[
\zeta+1 + \sum_{i=1}^{\zeta} (\zeta+2-i)[ (d-1)^i -
(d-1)^{i-1} ] = \sum_{i=0}^{\zeta} (d-1)^{i},
\]
which is the total number of vertices of the tree.
\end{pf}

The zeros of Chebyshev polynomials of order $k$ can be easily shown to
be given by
\[
\cos\biggl( \frac{j\pi}{k+1} \biggr),\qquad j=1,2,\ldots,k.
\]

Thus an interesting phenomenon transpires in this analysis. If we drop
the multiplicities and consider the empirical distribution of the
distinct eigenvalues, they are the zeros of the Chebyshev polynomials
of increasing order. These zeros are cosine transformations of
equidistant points on the unit circle; and hence their empirical
distribution converges to the arc-sine law. However, the entire
empirical spectral distribution converges (see~\cite{bordlelarge}) to
the spectral distribution of the infinite tree which is the semicircle
law. The effect of the multiplicities is strong enough to flip the
``smile'' of the arc-sine law to the ``frown'' of the semicircle! Also
note that the gap of the spectrum from $2$ is about twice of $\pi
^2/(\zeta+1)^2$, and does not depend on $d$.

Some facts about eigenvectors of this tree are also easy to derive and
might be also worthwhile to look at. For example, by the spectral
theorem, one can write
%
%
\begin{equation}\label{eigenvectree}
-\frac{U_{\zeta}(z/2)}{U_{\zeta+1}(z/2)}=\varphi(z) = \biggl( \frac
{1}{\sqrt{d-1}} H - z \biggr)^{-1}_{\Root, \Root} = \sum_{i} \frac
{\Vert P_i e_{\Root} \Vert^2}{\lambda_i - z}.
\end{equation}
Here the sum on the right goes over distinct eigenvalues of the
adjacency matrix and $P_i e_{\Root}$ refers to the projection of the
vector $e_{\Root}$ on the eigenspace corresponding to $\lambda_i$.

Notice that the above is a meromorphic function of $z$. From the
leftmost expression in (\ref{eigenvectree}), it is obvious that the
function has poles at (twice) the zeros of $U_{\zeta+1}$. It follows
then that $P_i e_{\Root}$ is zero for all eigenvalues except when
$\lambda_i$ is twice of a root of $U_{\zeta+1}$. However, these roots
are simple, as we show in the previous lemma. Hence, $\Vert P_i
e_{\Root} \Vert^2$ is precise the square of the ``root''-coordinate of the
$i$th eigenvector.

Its value can be easily computed. For any root $\lambda_i$ of
$U_{\zeta+1}$, we get
\[
\Vert P_i e_{\Root} \Vert^2 = \lim_{z \rightarrow2\lambda_i} (z -
2\lambda_i)\frac{U_{\zeta}(z/2)}{U_{\zeta+1}(z/2)} = \frac
{2U_{\zeta}(\lambda_i)}{U_{\zeta+1}'(\lambda_i)}.
\]
Here\vspace*{2pt} $U'_{\zeta+1}$ refers to the derivative of the polynomial
$U_{\zeta+1}$. Other coordinates can be similarly derived.
\end{appendix}

\section*{Acknowledgments}
It is our pleasure to thank Chris Burdzy,
Sourav Chatterjee, Manju Krishnapur, Nati Linial and Sasha Soshnikov
for very useful discussions. Both authors would like to thank Tobias
Johnson for a~very careful reading of this paper. Ioana is grateful to
MSRI for their hospitality during the Fall 2010 quarter, as part of the
program \textit{Random Matrix Theory}, \textit{Interacting Particle Systems and
Integrable Systems}, during which this work was completed.


%

%
\printaddresses


\begin{thebibliography}{70}

\bibitem{alonkrivvu}
\begin{barticle}[mr]
\bauthor{\bsnm{Alon},~\bfnm{Noga}\binits{N.}},
  \bauthor{\bsnm{Krivelevich},~\bfnm{Michael}\binits{M.}} \AND
  \bauthor{\bsnm{Vu},~\bfnm{Van~H.}\binits{V.~H.}}
(\byear{2002}).
\btitle{On the concentration of eigenvalues of random symmetric matrices}.
\bjournal{Israel J. Math.}
\bvolume{131}
\bpages{259--267}.
\bid{doi={10.1007/BF02785860}, issn={0021-2172}, mr={1942311}}%
\end{barticle}%
\endbibitem%
%
\bibitem{agzbook}
\begin{bbook}[mr]
\bauthor{\bsnm{Anderson},~\bfnm{Greg~W.}\binits{G.~W.}},
  \bauthor{\bsnm{Guionnet},~\bfnm{Alice}\binits{A.}} \AND
  \bauthor{\bsnm{Zeitouni},~\bfnm{Ofer}\binits{O.}}
(\byear{2010}).
\btitle{An Introduction to Random Matrices}.
\bseries{Cambridge Studies in Advanced Mathematics}
\bvolume{118}.
\bpublisher{Cambridge Univ. Press}, \baddress{Cambridge}.
\bid{mr={2760897}}
\bptnote{check year}%
\end{bbook}
\endbibitem

\bibitem{baiconv}
\begin{barticle}[mr]
\bauthor{\bsnm{Bai},~\bfnm{Z.~D.}\binits{Z.~D.}}
(\byear{1993}).
\btitle{Convergence rate of expected spectral distributions of large random
  matrices. {I}. {W}igner matrices}.
\bjournal{Ann. Probab.}
\bvolume{21}
\bpages{625--648}.
\bid{issn={0091-1798}, mr={1217559}}
\end{barticle}
\endbibitem

\bibitem{buymeapan}
\begin{barticle}[mr]
\bauthor{\bsnm{Bai},~\bfnm{Z.~D.}\binits{Z.~D.}},
  \bauthor{\bsnm{Miao},~\bfnm{B.~Q.}\binits{B.~Q.}} \AND
  \bauthor{\bsnm{Pan},~\bfnm{G.~M.}\binits{G.~M.}}
(\byear{2007}).
\btitle{On asymptotics of eigenvectors of large sample covariance matrix}.
\bjournal{Ann. Probab.}
\bvolume{35}
\bpages{1532--1572}.
\bid{doi={10.1214/009117906000001079}, issn={0091-1798}, mr={2330979}}
\end{barticle}
\endbibitem

\bibitem{baisil}
\begin{bbook}[auto:STB|2011-03-03|12:04:44]
\bauthor{\bsnm{Bai},~\bfnm{Z.~D.}\binits{Z.~D.}} \AND
  \bauthor{\bsnm{Silverstein},~\bfnm{J.~W.}\binits{J.~W.}}
(\byear{2006}).
\btitle{Spectral Analysis of Large Dimensional Random Matrices}.
\bseries{Mathematics Monograph Series}
\bvolume{2}.
\bpublisher{Science Press}, \baddress{Beijing}.
\end{bbook}
\endbibitem

\bibitem{baiyao05}
\begin{barticle}[mr]
\bauthor{\bsnm{Bai},~\bfnm{Z.~D.}\binits{Z.~D.}} \AND
  \bauthor{\bsnm{Yao},~\bfnm{J.}\binits{J.}}
(\byear{2005}).
\btitle{On the convergence of the spectral empirical process of {W}igner
  matrices}.
\bjournal{Bernoulli}
\bvolume{11}
\bpages{1059--1092}.
\bid{doi={10.3150/bj/1137421640}, issn={1350-7265}, mr={2189081}}
\end{barticle}
\endbibitem

\bibitem{baiksuidan05}
\begin{barticle}[mr]
\bauthor{\bsnm{Baik},~\bfnm{Jinho}\binits{J.}} \AND
  \bauthor{\bsnm{Suidan},~\bfnm{Toufic~M.}\binits{T.~M.}}
(\byear{2005}).
\btitle{A {GUE} central limit theorem and universality of directed first and
  last passage site percolation}.
\bjournal{Int. Math. Res. Not. IMRN}
\bvolume{6}
\bpages{325--337}.
\bid{doi={10.1155/IMRN.2005.325}, issn={1073-7928}, mr={2131383}}
\end{barticle}
\endbibitem

\bibitem{baiksuidan07}
\begin{barticle}[mr]
\bauthor{\bsnm{Baik},~\bfnm{Jinho}\binits{J.}} \AND
  \bauthor{\bsnm{Suidan},~\bfnm{Toufic~M.}\binits{T.~M.}}
(\byear{2007}).
\btitle{Random matrix central limit theorems for nonintersecting random walks}.
\bjournal{Ann. Probab.}
\bvolume{35}
\bpages{1807--1834}.
\bid{doi={10.1214/009117906000001105}, issn={0091-1798}, mr={2349576}}
\end{barticle}
\endbibitem

\bibitem{Forresterpoly}
\begin{barticle}[mr]
\bauthor{\bsnm{Baker},~\bfnm{T.~H.}\binits{T.~H.}} \AND
  \bauthor{\bsnm{Forrester},~\bfnm{P.~J.}\binits{P.~J.}}
(\byear{1997}).
\btitle{The {C}alogero--{S}utherland model and generalized classical
  polynomials}.
\bjournal{Comm. Math. Phys.}
\bvolume{188}
\bpages{175--216}.
\bid{doi={10.1007/s002200050161}, issn={0010-3616}, mr={1471336}}
\end{barticle}
\endbibitem

\bibitem{bauergolinelli01}
\begin{barticle}[mr]
\bauthor{\bsnm{Bauer},~\bfnm{M.}\binits{M.}} \AND
  \bauthor{\bsnm{Golinelli},~\bfnm{O.}\binits{O.}}
(\byear{2001}).
\btitle{Random incidence matrices: Moments of the spectral density}.
\bjournal{J. Stat. Phys.}
\bvolume{103}
\bpages{301--337}.
\bid{doi={10.1023/A:1004879905284}, issn={0022-4715}, mr={1828732}}
\end{barticle}
\endbibitem

\bibitem{benarouspeche}
\begin{barticle}[mr]
\bauthor{\bsnm{Ben~Arous},~\bfnm{G.}\binits{G.}} \AND
  \bauthor{\bsnm{P{\'e}ch{\'e}},~\bfnm{S.}\binits{S.}}
(\byear{2005}).
\btitle{Universality of local eigenvalue statistics for some sample covariance
  matrices}.
\bjournal{Comm. Pure Appl. Math.}
\bvolume{58}
\bpages{1316--1357}.
\bid{doi={10.1002/cpa.20070}, issn={0010-3640}, mr={2162782}}
\end{barticle}
\endbibitem

\bibitem{bhamidievanssen}
\begin{bmisc}[auto:STB|2011-03-03|12:04:44]
\bauthor{\bsnm{Bhamidi},~\bfnm{S.}\binits{S.}},
  \bauthor{\bsnm{Evans},~\bfnm{S.}\binits{S.}} \AND
  \bauthor{\bsnm{Sen},~\bfnm{A.}\binits{A.}}
(\byear{2008}).
\bhowpublished{Spectra of random large trees. Preprint. Available at
  \href{http://arxiv.org/abs/arXiv:0903.3589v2}{arXiv:0903.3589v2}}.
\end{bmisc}
\endbibitem


\bibitem{bordlelarge}
\begin{barticle}[mr]
\bauthor{\bsnm{Bordenave},~\bfnm{Charles}\binits{C.}} \AND
  \bauthor{\bsnm{Lelarge},~\bfnm{Marc}\binits{M.}}
(\byear{2010}).
\btitle{Resolvent of large random graphs}.
\bjournal{Random Structures Algorithms}
\bvolume{37}
\bpages{332--352}.
\bid{doi={10.1002/rsa.20313}, issn={1042-9832}, mr={2724665}}
\bptnote{check year}%
\end{barticle}
\endbibitem

\bibitem{pagerank}
\begin{bincollection}[auto:STB|2011-03-03|12:04:44]
\bauthor{\bsnm{Brin},~\bfnm{S.}\binits{S.}} \AND
  \bauthor{\bsnm{Lawrence},~\bfnm{P.}\binits{P.}}
(\byear{1998}).
\btitle{The anatomy of a large-scale hypertextual Web search engine}.
In \bbooktitle{Proceedings of the Seventh International Conference on World
  Wide Web 7}
\bpages{107--117}.
\bpublisher{Brisbane}, \baddress{Australia}.
\end{bincollection}
\endbibitem

\bibitem{brodershamir}
\begin{bincollection}[auto:STB|2011-03-03|12:04:44]
\bauthor{\bsnm{Broder},~\bfnm{A.}\binits{A.}} \AND
  \bauthor{\bsnm{Shamir},~\bfnm{E.}\binits{E.}}
(\byear{1987}).
\btitle{On the second eigenvalue of random regular graphs}.
In \bbooktitle{28th Annual Symposium on Foundations of Computer Science (Los
  Angeles, 1987)}
\bpages{286--294}.
\bpublisher{IEEE Comput. Soc. Press}, \baddress{Washington, DC}.
\end{bincollection}
\endbibitem

\bibitem{chatterjee}
\begin{barticle}[mr]
\bauthor{\bsnm{Chatterjee},~\bfnm{Sourav}\binits{S.}}
(\byear{2006}).
\btitle{A generalization of the {L}indeberg principle}.
\bjournal{Ann. Probab.}
\bvolume{34}
\bpages{2061--2076}.
\bid{doi={10.1214/009117906000000575}, issn={0091-1798}, mr={2294976}}
\end{barticle}
\endbibitem


\bibitem{deift06a}
\begin{bincollection}[mr]
\bauthor{\bsnm{Deift},~\bfnm{Percy}\binits{P.}}
(\byear{2007}).
\btitle{Universality for mathematical and physical systems}.
In \bbooktitle{International {C}ongress of {M}athematicians. {V}ol. {I}}
\bpages{125--152}.
\bpublisher{Eur. Math. Soc.}, \baddress{Z\"urich}.
\bid{doi={10.4171/022-1/7}, mr={2334189}}
\bptnote{check year}%
\end{bincollection}
\endbibitem

\bibitem{deiftgioev}
\begin{barticle}[mr]
\bauthor{\bsnm{Deift},~\bfnm{Percy}\binits{P.}} \AND
  \bauthor{\bsnm{Gioev},~\bfnm{Dimitri}\binits{D.}}
(\byear{2007}).
\btitle{Universality at the edge of the spectrum for unitary, orthogonal, and
  symplectic ensembles of random matrices}.
\bjournal{Comm. Pure Appl. Math.}
\bvolume{60}
\bpages{867--910}.
\bid{doi={10.1002/cpa.20164}, issn={0010-3640}, mr={2306224}}
\bptnote{check year}%
\end{barticle}
\endbibitem

\bibitem{dekelleelinial}
\begin{barticle}[auto:STB|2011-03-03|12:04:44]
\bauthor{\bsnm{Dekel},~\bfnm{Y.}\binits{Y.}},
  \bauthor{\bsnm{Lee},~\bfnm{J.~R.}\binits{J.~R.}} \AND
  \bauthor{\bsnm{Linial},~\bfnm{N.}\binits{N.}}
(\byear{2011}).
\btitle{Eigenvectors of random graphs: Nodal domains}.
\bjournal{Random Structures Algorithms}
\bvolume{39}
\bpages{3958}.
\end{barticle}
\endbibitem

\bibitem{dumitriu06}
\begin{barticle}[mr]
\bauthor{\bsnm{Dumitriu},~\bfnm{Ioana}\binits{I.}} \AND
  \bauthor{\bsnm{Edelman},~\bfnm{Alan}\binits{A.}}
(\byear{2006}).
\btitle{Global spectrum fluctuations for the {$\beta $}-{H}ermite and {$\beta
  $}-{L}aguerre ensembles via matrix models}.
\bjournal{J. Math. Phys.}
\bvolume{47}
\bpages{063302, 36}.
\bid{doi={10.1063/1.2200144}, issn={0022-2488}, mr={2239975}}
\end{barticle}
\endbibitem

\bibitem{simevec}
\begin{barticle}[mr]
\bauthor{\bsnm{Elon},~\bfnm{Yehonatan}\binits{Y.}}
(\byear{2008}).
\btitle{Eigenvectors of the discrete {L}aplacian on regular graphs---a
  statistical approach}.
\bjournal{J. Phys. A}
\bvolume{41}
\bpages{435203, 17}.
\bid{doi={10.1088/1751-8113/41/43/435203}, issn={1751-8113}, mr={2453165}}
\end{barticle}
\endbibitem

\bibitem{EKYY}
\begin{bmisc}[auto:STB|2011-03-03|12:04:44]
\bauthor{\bsnm{Erd{\H{o}}s},~\bfnm{L.}\binits{L.}},
  \bauthor{\bsnm{Knowles},~\bfnm{A.}\binits{A.}},
  \bauthor{\bsnm{Yau},~\bfnm{H.~T.}\binits{H.~T.}} \AND
  \bauthor{\bsnm{Yin},~\bfnm{J.}\binits{J.}}
(\byear{2011}).
\bhowpublished{Spectral statistics of Erd\H{o}s--R\'enyi graphs I: Local
  semicircle law. Preprint.
  \texttt{
  \href{http://arxiv.org/abs/1103.1919}{http://arxiv.org/abs/}
  \href{http://arxiv.org/abs/1103.1919}{1103.1919}}}.
\end{bmisc}
\endbibitem

\bibitem{eprsy}
\begin{barticle}[mr]
\bauthor{\bsnm{Erd{\H{o}}s},~\bfnm{L{\'a}szl{\'o}}\binits{L.}},
  \bauthor{\bsnm{P{\'e}ch{\'e}},~\bfnm{Sandrine}\binits{S.}},
  \bauthor{\bsnm{Ram{\'{\i}}rez},~\bfnm{Jos{\'e}~A.}\binits{J.~A.}},
  \bauthor{\bsnm{Schlein},~\bfnm{Benjamin}\binits{B.}} \AND
  \bauthor{\bsnm{Yau},~\bfnm{Horng-Tzer}\binits{H.-T.}}
(\byear{2010}).
\btitle{Bulk universality for {W}igner matrices}.
\bjournal{Comm. Pure Appl. Math.}
\bvolume{63}
\bpages{895--925}.
\bid{issn={0010-3640}, mr={2662426}}
\end{barticle}
\endbibitem

\bibitem{ERSTVY}
\begin{barticle}[mr]
\bauthor{\bsnm{Erd{\H{o}}s},~\bfnm{L{\'a}szl{\'o}}\binits{L.}},
  \bauthor{\bsnm{Ram{\'{\i}}rez},~\bfnm{Jos{\'e}}\binits{J.}},
  \bauthor{\bsnm{Schlein},~\bfnm{Benjamin}\binits{B.}},
  \bauthor{\bsnm{Tao},~\bfnm{Terence}\binits{T.}},
  \bauthor{\bsnm{Vu},~\bfnm{Van}\binits{V.}} \AND
  \bauthor{\bsnm{Yau},~\bfnm{Horng-Tzer}\binits{H.-T.}}
(\byear{2010}).
\btitle{Bulk universality for {W}igner {H}ermitian matrices with subexponential
  decay}.
\bjournal{Math. Res. Lett.}
\bvolume{17}
\bpages{667--674}.
\bid{issn={1073-2780}, mr={2661171}}
\end{barticle}
\endbibitem

\bibitem{esy1}
\begin{barticle}[mr]
\bauthor{\bsnm{Erd{\H{o}}s},~\bfnm{L{\'a}szl{\'o}}\binits{L.}},
  \bauthor{\bsnm{Schlein},~\bfnm{Benjamin}\binits{B.}} \AND
  \bauthor{\bsnm{Yau},~\bfnm{Horng-Tzer}\binits{H.-T.}}
(\byear{2009}).
\btitle{Semicircle law on short scales and delocalization of eigenvectors for
  {W}igner random matrices}.
\bjournal{Ann. Probab.}
\bvolume{37}
\bpages{815--852}.
\bid{doi={10.1214/08-AOP421}, issn={0091-1798}, mr={2537522}}
\end{barticle}
\endbibitem

\bibitem{esy2}
\begin{barticle}[mr]
\bauthor{\bsnm{Erd{\H{o}}s},~\bfnm{L{\'a}szl{\'o}}\binits{L.}},
  \bauthor{\bsnm{Schlein},~\bfnm{Benjamin}\binits{B.}} \AND
  \bauthor{\bsnm{Yau},~\bfnm{Horng-Tzer}\binits{H.-T.}}
(\byear{2009}).
\btitle{Local semicircle law and complete delocalization for {W}igner random
  matrices}.
\bjournal{Comm. Math. Phys.}
\bvolume{287}
\bpages{641--655}.
\bid{doi={10.1007/s00220-008-0636-9}, issn={0010-3616}, mr={2481753}}
\end{barticle}
\endbibitem

\bibitem{eyy}
\begin{bmisc}[auto:STB|2011-03-03|12:04:44]
\bauthor{\bsnm{Erd{\H{o}}s},~\bfnm{L.}\binits{L.}},
  \bauthor{\bsnm{Yau},~\bfnm{H.~T.}\binits{H.~T.}} \AND
  \bauthor{\bsnm{Yin},~\bfnm{J.}\binits{J.}}
(\byear{2010}).
\bhowpublished{Universality for generalized Wigner matrices with Bernoulli
  distribution. Preprint. Available at
  \href{http://arxiv.org/abs/arXiv:1003.3813v4}{arXiv:1003.3813v4}}.
\end{bmisc}
\endbibitem

\bibitem{feigeofek05}
\begin{barticle}[mr]
\bauthor{\bsnm{Feige},~\bfnm{Uriel}\binits{U.}} \AND
  \bauthor{\bsnm{Ofek},~\bfnm{Eran}\binits{E.}}
(\byear{2005}).
\btitle{Spectral techniques applied to sparse random graphs}.
\bjournal{Random Structures Algorithms}
\bvolume{27}
\bpages{251--275}.
\bid{doi={10.1002/rsa.20089}, issn={1042-9832}, mr={2155709}}
\end{barticle}
\endbibitem

\bibitem{friedmane2}
\begin{barticle}[mr]
\bauthor{\bsnm{Friedman},~\bfnm{Joel}\binits{J.}}
(\byear{1991}).
\btitle{On the second eigenvalue and random walks in random {$d$}-re\-gular
  graphs}.
\bjournal{Combinatorica}
\bvolume{11}
\bpages{331--362}.
\bid{doi={10.1007/BF01275669}, issn={0209-9683}, mr={1137767}}
\end{barticle}
\endbibitem

\bibitem{friedmannodal}
\begin{barticle}[mr]
\bauthor{\bsnm{Friedman},~\bfnm{Joel}\binits{J.}}
(\byear{1993}).
\btitle{Some geometric aspects of graphs and their eigenfunctions}.
\bjournal{Duke Math. J.}
\bvolume{69}
\bpages{487--525}.
\bid{doi={10.1215/S0012-7094-93-06921-9}, issn={0012-7094}, mr={1208809}}
\end{barticle}
\endbibitem

\bibitem{friedmanalon}
\begin{barticle}[mr]
\bauthor{\bsnm{Friedman},~\bfnm{Joel}\binits{J.}}
(\byear{2008}).
\btitle{A proof of {A}lon's second eigenvalue conjecture and related problems}.
\bjournal{Mem. Amer. Math. Soc.}
\bvolume{195}
\bpages{viii+100}.
\bid{issn={0065-9266}, mr={2437174}}
\end{barticle}
\endbibitem

\bibitem{furedikomlos}
\begin{barticle}[mr]
\bauthor{\bsnm{F{\"u}redi},~\bfnm{Z.}\binits{Z.}} \AND
  \bauthor{\bsnm{Koml{\'o}s},~\bfnm{J.}\binits{J.}}
(\byear{1981}).
\btitle{The eigenvalues of random symmetric matrices}.
\bjournal{Combinatorica}
\bvolume{1}
\bpages{233--241}.
\bid{doi={10.1007/BF02579329}, issn={0209-9683}, mr={0637828}}
\end{barticle}
\endbibitem

\bibitem{guionnetzeitouni}
\begin{barticle}[mr]
\bauthor{\bsnm{Guionnet},~\bfnm{A.}\binits{A.}} \AND
  \bauthor{\bsnm{Zeitouni},~\bfnm{O.}\binits{O.}}
(\byear{2000}).
\btitle{Concentration of the spectral measure for large matrices}.
\bjournal{Electron. Commun. Probab.}
\bvolume{5}
\bpages{119--136 (electronic)}.
\bid{issn={1083-589X}, mr={1781846}}
\end{barticle}
\endbibitem

\bibitem{simeval}
\begin{bincollection}[mr]
\bauthor{\bsnm{Jakobson},~\bfnm{Dmitry}\binits{D.}},
  \bauthor{\bsnm{Miller},~\bfnm{Stephen~D.}\binits{S.~D.}},
  \bauthor{\bsnm{Rivin},~\bfnm{Igor}\binits{I.}} \AND
  \bauthor{\bsnm{Rudnick},~\bfnm{Ze{\'e}v}\binits{Z.}}
(\byear{1999}).
\btitle{Eigenvalue spacings for regular graphs}.
In \bbooktitle{Emerging Applications of Number Theory ({M}inneapolis, {MN},
  1996)}.
\bseries{IMA Vol. Math. Appl.}
\bvolume{109}
\bpages{317--327}.
\bpublisher{Springer}, \baddress{New York}.
\bid{mr={1691538}}
\bptnote{check year}%
\end{bincollection}
\endbibitem

\bibitem{johanssoncltherm}
\begin{barticle}[mr]
\bauthor{\bsnm{Johansson},~\bfnm{Kurt}\binits{K.}}
(\byear{1998}).
\btitle{On fluctuations of eigenvalues of random {H}ermitian matrices}.
\bjournal{Duke Math. J.}
\bvolume{91}
\bpages{151--204}.
\bid{doi={10.1215/S0012-7094-98-09108-6}, issn={0012-7094}, mr={1487983}}
\end{barticle}
\endbibitem

\bibitem{kesten}
\begin{barticle}[mr]
\bauthor{\bsnm{Kesten},~\bfnm{Harry}\binits{H.}}
(\byear{1959}).
\btitle{Symmetric random walks on groups}.
\bjournal{Trans. Amer. Math. Soc.}
\bvolume{92}
\bpages{336--354}.
\bid{issn={0002-9947}, mr={0109367}}
\end{barticle}
\endbibitem

\bibitem{khorunzhy01}
\begin{barticle}[mr]
\bauthor{\bsnm{Khorunzhy},~\bfnm{A.}\binits{A.}}
(\byear{2001}).
\btitle{Sparse random matrices: Spectral edge and statistics of rooted trees}.
\bjournal{Adv. in Appl. Probab.}
\bvolume{33}
\bpages{124--140}.
\bid{doi={10.1239/aap/999187900}, issn={0001-8678}, mr={1825319}}
\end{barticle}
\endbibitem

\bibitem{khokhopastur}
\begin{barticle}[mr]
\bauthor{\bsnm{Khorunzhy},~\bfnm{Alexei~M.}\binits{A.~M.}},
  \bauthor{\bsnm{Khoruzhenko},~\bfnm{Boris~A.}\binits{B.~A.}} \AND
  \bauthor{\bsnm{Pastur},~\bfnm{Leonid~A.}\binits{L.~A.}}
(\byear{1996}).
\btitle{Asymptotic properties of large random matrices with independent
  entries}.
\bjournal{J. Math. Phys.}
\bvolume{37}
\bpages{5033--5060}.
\bid{doi={10.1063/1.531589}, issn={0022-2488}, mr={1411619}}
\end{barticle}
\endbibitem

\bibitem{krivsudakov03}
\begin{barticle}[mr]
\bauthor{\bsnm{Krivelevich},~\bfnm{Michael}\binits{M.}} \AND
  \bauthor{\bsnm{Sudakov},~\bfnm{Benny}\binits{B.}}
(\byear{2003}).
\btitle{The largest eigenvalue of sparse random graphs}.
\bjournal{Combin. Probab. Comput.}
\bvolume{12}
\bpages{61--72}.
\bid{doi={10.1017/S0963548302005424}, issn={0963-5483}, mr={1967486}}
\end{barticle}
\endbibitem

\bibitem{LW}
\begin{bbook}[mr]
\bauthor{\bsnm{Lorentzen},~\bfnm{Lisa}\binits{L.}} \AND
  \bauthor{\bsnm{Waadeland},~\bfnm{Haakon}\binits{H.}}
(\byear{2008}).
\btitle{Continued Fractions. {V}ol. 1. Convergence Theory},
\bedition{2nd} ed.
\bseries{Atlantis Studies in Mathematics for Engineering and Science}
\bvolume{1}.
\bpublisher{Atlantis Press}, \baddress{Paris}.
\bid{mr={2433845}}
\end{bbook}
\endbibitem

\bibitem{LPS}
\begin{barticle}[mr]
\bauthor{\bsnm{Lubotzky},~\bfnm{A.}\binits{A.}},
  \bauthor{\bsnm{Phillips},~\bfnm{R.}\binits{R.}} \AND
  \bauthor{\bsnm{Sarnak},~\bfnm{P.}\binits{P.}}
(\byear{1988}).
\btitle{Ramanujan graphs}.
\bjournal{Combinatorica}
\bvolume{8}
\bpages{261--277}.
\bid{doi={10.1007/BF02126799}, issn={0209-9683}, mr={0963118}}
\end{barticle}
\endbibitem

\bibitem{MH}
\begin{bbook}[mr]
\bauthor{\bsnm{Mason},~\bfnm{J.~C.}\binits{J.~C.}} \AND
  \bauthor{\bsnm{Handscomb},~\bfnm{D.~C.}\binits{D.~C.}}
(\byear{2003}).
\btitle{Chebyshev Polynomials}.
\bpublisher{Chapman \& Hall/CRC}, \baddress{Boca Raton, FL}.
\bid{mr={1937591}}
\end{bbook}
\endbibitem

\bibitem{mckay81}
\begin{barticle}[mr]
\bauthor{\bsnm{McKay},~\bfnm{Brendan~D.}\binits{B.~D.}}
(\byear{1981}).
\btitle{The expected eigenvalue distribution of a large regular graph}.
\bjournal{Linear Algebra Appl.}
\bvolume{40}
\bpages{203--216}.
\bid{doi={10.1016/0024-3795(81)90150-6}, issn={0024-3795}, mr={0629617}}
\end{barticle}
\endbibitem


\bibitem{MW}
\begin{barticle}[mr]
\bauthor{\bsnm{McKay},~\bfnm{Brendan~D.}\binits{B.~D.}} \AND
  \bauthor{\bsnm{Wormald},~\bfnm{Nicholas~C.}\binits{N.~C.}}
(\byear{1997}).
\btitle{The degree sequence of a random graph. {I}. {T}he models}.
\bjournal{Random Structures Algorithms}
\bvolume{11}
\bpages{97--117}.
\bid{doi={10.1002/(SICI)1098-2418(199709)11:2\&lt;97::AID-RSA1\&gt;3.3.CO;2-E},
  issn={1042-9832}, mr={1610253}}
\end{barticle}
\endbibitem

\bibitem{mww04}
\begin{barticle}[mr]
\bauthor{\bsnm{McKay},~\bfnm{Brendan~D.}\binits{B.~D.}},
  \bauthor{\bsnm{Wormald},~\bfnm{Nicholas~C.}\binits{N.~C.}} \AND
  \bauthor{\bsnm{Wysocka},~\bfnm{Beata}\binits{B.}}
(\byear{2004}).
\btitle{Short cycles in random regular graphs}.
\bjournal{Electron. J. Combin.}
\bvolume{11}
\bpages{12 pp. (electronic)}.
\bid{issn={1077-8926}, mr={2097332}}
\end{barticle}
\endbibitem

\bibitem{mehtabook}
\begin{bbook}[mr]
\bauthor{\bsnm{Mehta},~\bfnm{Madan~Lal}\binits{M.~L.}}
(\byear{1991}).
\btitle{Random Matrices}, \bedition{2nd} ed.
\bpublisher{Academic Press}, \baddress{Boston, MA}.
\bid{mr={1083764}}
\end{bbook}
\endbibitem

\bibitem{meilashi}
\begin{bincollection}[auto:STB|2011-03-03|12:04:44]
\bauthor{\bsnm{Meila},~\bfnm{M.}\binits{M.}} \AND
  \bauthor{\bsnm{Shi},~\bfnm{J.}\binits{J.}}
(\byear{2011}).
\btitle{Learning segmentation with random walks}.
In \bbooktitle{Adv. Neural Inf. Process. Syst.}
\bvolume{14}
\bpages{873--879}.
\bpublisher{MIT Press}, \baddress{Cambridge, MA}.
\end{bincollection}
\endbibitem

\bibitem{mirlinfyod91}
\begin{barticle}[mr]
\bauthor{\bsnm{Mirlin},~\bfnm{A.~D.}\binits{A.~D.}} \AND
  \bauthor{\bsnm{Fyodorov},~\bfnm{Yan~V.}\binits{Y.~V.}}
(\byear{1991}).
\btitle{Universality of level correlation function of sparse random matrices}.
\bjournal{J. Phys. A}
\bvolume{24}
\bpages{2273--2286}.
\bid{issn={0305-4470}, mr={1118532}}
\end{barticle}
\endbibitem

\bibitem{morgenstern}
\begin{barticle}[mr]
\bauthor{\bsnm{Morgenstern},~\bfnm{Moshe}\binits{M.}}
(\byear{1994}).
\btitle{Existence and explicit constructions of {$q+1$} regular {R}amanujan
  graphs for every prime power {$q$}}.
\bjournal{J. Combin. Theory Ser. B}
\bvolume{62}
\bpages{44--62}.
\bid{doi={10.1006/jctb.1994.1054}, issn={0095-8956}, mr={1290630}}
\end{barticle}
\endbibitem

\bibitem{muirhead82a}
\begin{bbook}[mr]
\bauthor{\bsnm{Muirhead},~\bfnm{Robb~J.}\binits{R.~J.}}
(\byear{1982}).
\btitle{Aspects of Multivariate Statistical Theory}.
\bpublisher{Wiley}, \baddress{New York}.
\bid{mr={0652932}}
\end{bbook}
\endbibitem

\bibitem{pothensimonliou}
\begin{barticle}[mr]
\bauthor{\bsnm{Pothen},~\bfnm{Alex}\binits{A.}},
  \bauthor{\bsnm{Simon},~\bfnm{Horst~D.}\binits{H.~D.}} \AND
  \bauthor{\bsnm{Liou},~\bfnm{Kang-Pu}\binits{K.-P.}}
(\byear{1990}).
\btitle{Partitioning sparse matrices with eigenvectors of graphs}.
\bjournal{SIAM J. Matrix Anal. Appl.}
\bvolume{11}
\bpages{430--452}.
\bid{doi={10.1137/0611030}, issn={0895-4798}, mr={1054210}}
\end{barticle}
\endbibitem

\bibitem{semerjian02}
\begin{barticle}[mr]
\bauthor{\bsnm{Semerjian},~\bfnm{Guilhem}\binits{G.}} \AND
  \bauthor{\bsnm{Cugliandolo},~\bfnm{Leticia~F.}\binits{L.~F.}}
(\byear{2002}).
\btitle{Sparse random matrices: The eigenvalue spectrum revisited}.
\bjournal{J. Phys. A}
\bvolume{35}
\bpages{4837--4851}.
\bid{doi={10.1088/0305-4470/35/23/303}, issn={0305-4470}, mr={1916090}}
\end{barticle}
\endbibitem


\bibitem{shimalik00}
\begin{barticle}[auto:STB|2011-03-03|12:04:44]
\bauthor{\bsnm{Shi},~\bfnm{J.}\binits{J.}} \AND
  \bauthor{\bsnm{Malik},~\bfnm{J.}\binits{J.}}
(\byear{2000}).
\btitle{Normalized cuts and image segmentation}.
\bjournal{IEEE Transactions on Pattern Analysis and Machine Intelligence}
\bvolume{22}
\bpages{888--905}.
\end{barticle}
\endbibitem

\bibitem{soshuniv}
\begin{barticle}[mr]
\bauthor{\bsnm{Soshnikov},~\bfnm{Alexander}\binits{A.}}
(\byear{1999}).
\btitle{Universality at the edge of the spectrum in {W}igner random matrices}.
\bjournal{Comm. Math. Phys.}
\bvolume{207}
\bpages{697--733}.
\bid{doi={10.1007/s002200050743}, issn={0010-3616}, mr={1727234}}
\end{barticle}
\endbibitem


\bibitem{taovu2}
\begin{barticle}[mr]
\bauthor{\bsnm{Tao},~\bfnm{Terence}\binits{T.}} \AND
  \bauthor{\bsnm{Vu},~\bfnm{Van}\binits{V.}}
(\byear{2010}).
\btitle{Random matrices: Universality of local eigenvalue statistics up to the
  edge}.
\bjournal{Comm. Math. Phys.}
\bvolume{298}
\bpages{549--572}.
\bid{doi={10.1007/s00220-010-1044-5}, issn={0010-3616}, mr={2669449}}
\bptnote{check year}%
\end{barticle}
\endbibitem

\bibitem{taovu}
\begin{bmisc}[auto:STB|2011-03-03|12:04:44]
\bauthor{\bsnm{Tao},~\bfnm{T.}\binits{T.}} \AND
  \bauthor{\bsnm{Vu},~\bfnm{V.}\binits{V.}}
(\byear{2011}).
\bhowpublished{Random matrices: Universality of local eigenvalue statistics.
  \textit{Acta Math.} To appear. Available at
  \href{http://arxiv.org/abs/arXiv:0906.0510v10}{arXiv:0906.0510v10}}.
\end{bmisc}
\endbibitem

\bibitem{twuniv1}
\begin{bincollection}[mr]
\bauthor{\bsnm{Tracy},~\bfnm{Craig~A.}\binits{C.~A.}} \AND
  \bauthor{\bsnm{Widom},~\bfnm{Harold}\binits{H.}}
(\byear{1999}).
\btitle{Universality of the distribution functions of random matrix theory}.
In \bbooktitle{Statistical Physics on the Eve of the 21st Century}.
\bseries{Series on Advances in Statistical Mechanics}
\bvolume{14}
\bpages{230--239}.
\bpublisher{World Sci. Publishing}, \baddress{River Edge, NJ}.
\bid{mr={1704004}}
\end{bincollection}
\endbibitem

\bibitem{twuniv2}
\begin{bincollection}[mr]
\bauthor{\bsnm{Tracy},~\bfnm{Craig~A.}\binits{C.~A.}} \AND
  \bauthor{\bsnm{Widom},~\bfnm{Harold}\binits{H.}}
(\byear{2000}).
\btitle{Universality of the distribution functions of random matrix theory}.
In \bbooktitle{Integrable Systems: From Classical to Quantum ({M}ontr\'eal,
  {QC}, 1999)}.
\bseries{CRM Proceedings \& Lecture Notes}
\bvolume{26}
\bpages{251--264}.
\bpublisher{Amer. Math. Soc.}, \baddress{Providence, RI}.
\bid{mr={1791893}}
\end{bincollection}
\endbibitem

\bibitem{twl}
\begin{bincollection}[mr]
\bauthor{\bsnm{Tracy},~\bfnm{C.~A.}\binits{C.~A.}} \AND
  \bauthor{\bsnm{Widom},~\bfnm{H.}\binits{H.}}
(\byear{2000}).
\btitle{The distribution of the largest eigenvalue in the Gaussian ensembles}.
In \bbooktitle{Calogero--Moser--Sutherland Models}
(\beditor{\bfnm{J.~F.}\binits{J.~F.}~\bsnm{vanDiejen}} \AND
  \beditor{\bfnm{L.}\binits{L.}~\bsnm{Vinet}}, eds.).
\bseries{CRM Series in Mathematical Physics}
\bvolume{4}
\bpages{461--472}.
\bpublisher{Springer}, \baddress{New York}.
\end{bincollection}
\endbibitem

\bibitem{TVW}
\begin{bmisc}[mr]
\bauthor{\bsnm{Tran},~\bfnm{L.}\binits{L.}},
  \bauthor{\bsnm{Vu},~\bfnm{Van}\binits{V.}} \AND
  \bauthor{\bsnm{Wang},~\bfnm{K.}\binits{K.}}
(\byear{2010}).
\bhowpublished{Sparse random graphs: Eigenvalues and eigenvectors. Preprint.
  Available at
  \texttt{\href{http://www.math.rutgers.edu/\textasciitilde linhtran/}{http://www.math.rutgers.edu/}
  \href{http://www.math.rutgers.edu/\textasciitilde linhtran/}{\textasciitilde linhtran/}}.}
\end{bmisc}
\endbibitem


\bibitem{weiss99}
\begin{bincollection}[auto:STB|2011-03-03|12:04:44]
\bauthor{\bsnm{Weiss},~\bfnm{Y.}\binits{Y.}}
(\byear{1999}).
\btitle{Segmentation using eigenvectors: A unifying view}.
In \bbooktitle{Proc. of the 7th IEEE Intern. Conf. on
Computer Vision, Vol. 2}
\bpages{975--982}.
\bpublisher{IEEE Press}, \baddress{New York}.
\end{bincollection}
\endbibitem




\end{thebibliography}
\end{document}